\documentclass[a4paper,reqno]{amsart}

\usepackage{amssymb}
\usepackage{amstext}
\usepackage{amsmath}
\usepackage{amscd}
\usepackage{amsthm}
\usepackage{amsfonts}
\usepackage{enumerate}
\usepackage{graphicx}
\usepackage{latexsym}
\usepackage{mathrsfs}
\usepackage{mathtools}
\usepackage[all]{xy}
\xyoption{all}

\usepackage{pstricks}
\usepackage{lscape}
\usepackage{comment}

\newtheorem{theorem}{Theorem}[section]

\newtheorem{corollary}[theorem]{Corollary}
\newtheorem{lemma}[theorem]{Lemma}
\newtheorem{proposition}[theorem]{Proposition}
\newtheorem{definition-proposition}[theorem]{Definition-Proposition}
\newtheorem{problem}[theorem]{Problem}

\theoremstyle{definition}
\newtheorem{definition}[theorem]{Definition}

\newcommand{\pp}{{\mathfrak{p}}}

\newcommand{\DD}{\mathcal{D}}
\newcommand{\DDD}{\mathsf{D}}
\newcommand{\KKK}{\mathsf{K}}

\newcommand{\TT}{\mathcal{T}}
\newcommand{\UU}{\mathcal{U}}
\newcommand{\VV}{\mathcal{V}}
\newcommand{\XX}{\mathcal{X}}
\newcommand{\YY}{\mathcal{Y}}
\newcommand{\ZZ}{\mathcal{Z}}

\newcommand{\Z}{\mathbb{Z}}
\renewcommand{\P}{\mathbb{P}}

\newcommand{\X}{\mathbb{X}}
\renewcommand{\S}{\mathbb{S}}
\newcommand{\bo}{\operatorname{b}\nolimits}
\newcommand{\sg}{\operatorname{sg}\nolimits}

\newcommand{\projdim}{\mathop{{\rm proj.dim}}\nolimits}
\newcommand{\injdim}{\mathop{{\rm inj.dim}}\nolimits}

\newcommand{\Ext}{\operatorname{Ext}\nolimits}

\newcommand{\Hom}{\operatorname{Hom}\nolimits}
\newcommand{\End}{\operatorname{End}\nolimits}

\newcommand{\gl}{\mathop{{\rm gl.dim\,}}\nolimits}
\newcommand{\op}{\operatorname{op}\nolimits}
\newcommand{\RHom}{\mathbf{R}\strut\kern-.2em\operatorname{Hom}\nolimits}
\newcommand{\Lotimes}{\mathop{\stackrel{\mathbf{L}}{\otimes}}\nolimits}

\newcommand{\Spec}{\operatorname{Spec}\nolimits}
\newcommand{\Supp}{\operatorname{Supp}\nolimits}

\DeclareMathOperator{\moduleCategory}{\mathsf{mod}} \renewcommand{\mod}{\moduleCategory}
\DeclareMathOperator{\Mod}{\mathsf{Mod}}
\DeclareMathOperator{\proj}{\mathsf{proj}}

\DeclareMathOperator{\inj}{\mathsf{inj}}

\DeclareMathOperator{\thick}{\mathsf{thick}}

\DeclareMathOperator{\coh}{\mathsf{coh}}
\DeclareMathOperator{\per}{\mathsf{per}}

\DeclareMathOperator{\qgr}{\mathsf{qgr}}
\DeclareMathOperator{\CM}{\mathsf{CM}}

\DeclareMathOperator{\add}{\mathsf{add}}

\newcommand{\cut}{\ar@{-}@[|(5)]}

\numberwithin{equation}{section}

\begin{document}
\title[Tilting theory for Gorenstein rings in dimension one]{Tilting theory for Gorenstein rings in dimension one}

\author[Buchweitz]{Ragnar-Olaf Buchweitz$^\dagger$}
\address{R. Buchweitz: Department of Computer and Mathematical Sciences, University of Toronto Scarborough, Toronto, Ontario, Canada M1C 1A4}

\author[Iyama]{Osamu Iyama}
\address{O. Iyama: Graduate School of Mathematics, Nagoya University, Furocho, Chikusaku, Nagoya 464-8602, Japan}
\email{iyama@math.nagoya-u.ac.jp}
\urladdr{http://www.math.nagoya-u.ac.jp/~iyama/}

\author[Yamaura]{Kota Yamaura}
\address{K. Yamaura: Graduate Faculty of Interdisciplinary Research, Faculty of Engineering, University of Yamanashi, Takeda, Kofu 400-8510, Japan}
\email{kyamaura@yamanashi.ac.jp}

\thanks{${}^\dagger$The first author passed away on November 11th, 2017.}
\thanks{The second author was partially supported by JSPS Grant-in-Aid for Scientific Research (B) 24340004, (B) 16H03923, (C) 23540045 and (S) 15H05738. The third author was partially supported by Grant-in-Aid for JSPS Research Fellow 13J01461 and JSPS Grant-in-Aid for Young Scientists (B) 26800007.}
\thanks{2010 {\em Mathematics Subject Classification.} 16G50, 13C14,
16E35, 18E30, 14F05}
\thanks{{\em Key words and phrases.} Cohen-Macaulay module, Gorenstein ring, derived category, singularity category, triangulated category, Frobenius category, differential graded algebra, tilting theory}

\begin{abstract}
In representation theory, commutative algebra and algebraic geometry, it is an important problem to understand when the triangulated category $\DDD_{\sg}^{\Z}(R)=\underline{\CM}_0^{\Z}R$ admits a tilting (respectively, silting) object for a $\Z$-graded commutative Gorenstein ring $R=\bigoplus_{i\ge0}R_i$. Here $\DDD_{\sg}^{\Z}(R)$ is the singularity category, and $\underline{\CM}_0^{\Z}R$ is the stable category of $\Z$-graded Cohen-Macaulay $R$-modules which are locally free at all non-maximal prime ideals of $R$.

In this paper, we give a complete answer to this problem in the case $\dim R=1$ and $R_0$ is a field: We prove that $\underline{\CM}_0^{\Z}R$ always admits a silting object,
and that $\underline{\CM}_0^{\Z}R$ admits a tilting object if and only if either $R$ is regular or the $a$-invariant of $R$ is non-negative. Our silting/tilting object will be given explicitly.
We also show that, if $R$ is reduced and non-regular, then its $a$-invariant is non-negative and the above tilting object gives a full strong exceptional collection in $\underline{\CM}_0^{\Z}R=\underline{\CM}^{\Z}R$.
\end{abstract}
\maketitle

\section{Introduction}

\subsection{Background}

The study of maximal Cohen-Macaulay (CM) modules is one of the central subjects in commutative algebra and representation theory \cite{Au,CR,LW,Si,Y}.
When the ring $R$ is Gorenstein, the category
\[\CM R=\{X\in\mod R\mid\Ext^i_R(X,R)=0\ \mbox{ for all $i\ge1$}\}\]
of CM $R$-modules forms a Frobenius category and therefore its stable category $\underline{\CM} R$ has a natural structure of a triangulated category \cite{Ha1}.
The Verdier quotient $\DDD_{\sg}(R)=\DDD^{\bo}(\mod R)/\KKK^{\bo}(\proj R)$ introduced by Buchweitz \cite{Buc} and Orlov \cite{O1}, is canonically triangle equivalent to $\underline{\CM} R$, and hence is enhanced by the Frobenius category $\CM R$.
When $R$ is a hypersurface, it is also triangle equivalent to the stable category of matrix factorizations \cite{E}. It has increasing importance in algebraic geometry and physics.

Tilting theory controls triangle equivalences between derived categories of rings, and plays a significant role in various areas of mathematics (see e.g.\ \cite{AHK}).
Tilting theory also gives a powerful tool to study the stable categories of Gorenstein rings.
For example, for a finite dimensional algebra $\Lambda$ of finite global dimension, there is a triangle equivalence 
\begin{equation}\label{happel}
\underline{\mod}^{\Z}T(\Lambda)\simeq\KKK^{\bo}(\proj\Lambda)
\end{equation}
for the stable category $\underline{\mod}^{\Z}T(\Lambda)$ of the $\Z$-graded modules over the trivial extension algebra $T(\Lambda)$ \cite{Ha1}.
This is an important result which gives a large family of representation-finite self-injective algebras (see e.g.\ \cite{Sk}).
The second classical example is a triangle equivalence
\begin{equation*}\label{bbgg}
\underline{\mod}^{\Z}\bigwedge(k^n)\simeq\KKK^{\bo}(\proj\Lambda)
\end{equation*}
for the exterior algebra $\bigwedge(k^n)$ and the Beilinson algebra $\Lambda$ \cite{Be,BGG}.
The third classical example is a triangle equivalence 
\begin{equation}\label{simple dynkin}
\underline{\CM}^{\Z}R\simeq\KKK^{\bo}(\mod kQ)
\end{equation}
for the stable category of $\Z$-graded Cohen-Macaulay modules over a $\Z$-graded simple surface singularity $R$ and the path algebra $kQ$ of the Dynkin quiver $Q$ of the same type \cite{GL1,GL2,KST1}.
Each of the above triangle equivalences follows from the fact that the stable category has a \emph{tilting object}  (see Definition \ref{define tilting}).
In fact, under mild assumptions, a triangulated category admits a tilting object if and only if it is triangle equivalent to $\KKK^{\bo}(\proj\Lambda)$ for some ring $\Lambda$ (Proposition \ref{tilting functor}).
Recently, the class of silting objects was introduced to complete the class of tilting objects in the study of t-structures \cite{KV} and mutation \cite{AI}.
We will see that they also plays an important role in the study of the stable categories of Gorenstein rings.

It is well known in Cohen-Macaulay representation theory that the subcategory
\[\CM_0R=\{X\in\CM R\mid X_\pp\in\proj R_\pp\ \mbox{for all $\pp\in\Spec R$ with $\dim R_\pp<\dim R$}\}\]
behaves much nicer than $\CM R$ since it enjoys Auslander-Reiten-Serre duality, and hence it has almost split sequences if $R$ is complete local \cite{Au,Y} (cf.\ Proposition \ref{AR duality for CM}).
Therefore, for a $\Z$-graded Gorenstein ring $R$, we consider the Frobenius category
\begin{eqnarray}\label{CM0}
\CM_0^{\Z}R:=\{X\in\mod^{\Z}R\mid X\in\CM_0R\ \mbox{as an ungraded $R$-module}\}.
\end{eqnarray}
There are a number of $\Z$-graded Gorenstein rings $R$ such that the stable categories $\underline{\CM}_0^{\Z}R$ admit tilting objects, see e.g.\ \cite{AIR,DL1,DL2,FU,G1,G2,HIMO,HU,IO,IT,JKS,KST1,KST2,Ki1,Ki2,KLM,LP,LZ,MU,SV,U1,U2,Ya} and a survey article \cite{I}.
The following problem is important in representation theory, commutative algebra and algebraic geometry.

\begin{problem}\label{existence question}
Let $R=\bigoplus_{i\ge0}R_i$ be a $\Z$-graded Gorenstein ring such that $R_0$ is a field.
When does the stable category $\underline{\CM}^{\Z}_0R$ of $\Z$-graded Cohen-Macaulay $R$-modules have a tilting object?
\end{problem}

When $\dim R=0$, $\underline{\CM}^{\Z}_0R=\underline{\mod}^{\Z}R$ always has a tilting object. In fact, the third author gave a much more general result \cite{Ya} which also implies the triangle equivalence \eqref{happel} as a special case.

The aim of this paper is to give a complete answer to Problem \ref{existence question} when $\dim R=1$. Surprisingly to us, it is determined by the $a$-invariant of $R$. Our results are summarized as follows.

\begin{theorem}[Theorems \ref{general tilting} and \ref{negative case}]\label{summary}
Let $R=\bigoplus_{i\ge0}R_i$ be a Gorenstein ring in dimension one such that $R_0$ is a field.
Then $\underline{\CM}_0^{\Z}R$ always has a silting object.
Moreover, $\underline{\CM}_0^{\Z}R$ has a tilting object if and only if either $R$ is regular or the $a$-invariant of $R$ is non-negative.
\end{theorem}

In particular, the Grothendieck group $K_0(\underline{\CM}_0^{\Z}R)$ is a free abelian group of finite rank (Corollary \ref{grothendieck}).
To prove Theorem \ref{summary}, we interpret $\underline{\CM}_0^{\Z}R$ as a thick subcategory of the singularity category $\DDD_{\sg}^{\Z}(R)=\DDD^{\bo}(\mod^{\Z}R)/\KKK^{\bo}(\proj^{\Z}R)$ (Proposition \ref{subcategory of Dsg})
and give analogues of Orlov's semi-orthogonal decompositions \cite{O} of $\DDD^{\bo}(\mod^{\Z}R)$ (Theorem \ref{embedding2}).

\subsection{Our results}
Throughout this subsection, we assume the following.
\begin{itemize}
\item[(R1)] $R$ is a $\Z$-graded commutative Gorenstein ring of Krull dimension one. 
\item[(R2)] $R=\bigoplus_{i\ge0}R_i$ and $k:=R_0$ is a field.
\end{itemize}
Let $S$ be the set of all homogeneous non-zero-divisors in $R$,
and $K:=RS^{-1}$ the $\Z$-graded total quotient ring of $R$.
There exists then an integer $p>0$ such that $K(p)\simeq K$ as a graded $R$-module (Lemma \ref{basic properties of K}(b)).
Moreover, $\dim R=1$ implies that $K= R[r^{-1}]$ holds for each homogeneous non-zero-divisor $r$ of positive degree (Lemma \ref{basic properties of K1}).

Let $\mod^{\Z}R$ be the category of $\Z$-graded finitely generated $R$-modules, by $\mod_0^{\Z}R$ the category of $\Z$-graded $R$-modules of finite length, and by $\proj^{\Z}R$ the category of $\Z$-graded finitely generated projective $R$-modules.
For $X\in\mod^{\Z}R$ and $n\in\Z$, let
\[X_{\ge n}=X_{>n-1}:=\bigoplus_{i\ge n}X_i.\]
Let $\qgr R=\mod^{\Z}R/\mod_0^{\Z}R$ be the quotient category. This is equivalent to the category of coherent sheaves on the quotient stack $[(\Spec R\setminus\{R_{>0}\})/k^*]$ \cite[Proposition 2.17]{O}.
Let $\DDD^{\bo}(\qgr R)$ be the bounded derived category of $\qgr R$, and 
let $\per(\qgr R)$ be its thick subcategory generated by $\proj^{\Z}R$.
Our starting point is the following result on the geometric side, where we refer to \cite{Hu} for the notion of exceptional collections.

\begin{theorem}\label{general tilting for qgr}
Under the setting (R1) and (R2), the following holds true.
\begin{enumerate}[\rm(a)]
\item $\qgr R$ has a progenerator $U:=\bigoplus_{i=1}^{p}K(i)_{\ge0}= \bigoplus_{i=1}^{p}K_{\ge i}(i)$, and $\per(\qgr R)$ has a tilting object $U$.
\item We have an equivalence $\qgr R\simeq\mod\Lambda$ and a triangle equivalence $\per(\qgr R)\simeq\KKK^{\bo}(\proj \Lambda)$ for $\Lambda:=
\End_{\qgr R}(U)$.
\item We have
{\small\begin{equation}\label{lambda}
\Lambda\simeq\End^{\Z}_R(U)=\begin{bmatrix}
K_0&K_{-1}&\cdots&K_{2-p}&K_{1-p}\\
K_{1}&K_0&\cdots&K_{3-p}&K_{2-p}\\
\vdots&\vdots&\ddots&\vdots&\vdots\\
K_{p-2}&K_{p-3}&\cdots&K_{0}&K_{-1}\\
K_{p-1}&K_{p-2}&\cdots&K_{1}&K_0
\end{bmatrix}.\end{equation}}
\item $\Lambda$ is a finite dimensional self-injective $k$-algebra.
\item If $R$ is reduced, then $\Lambda$ is a semisimple $k$-algebra. Otherwise $\Lambda$ has infinite global dimension.
\item If $R$ is reduced, then any ordering in the isomorphism classes of indecomposable direct summands of $U$ gives a full strong exceptional collection in $\per(\qgr R)$. Otherwise, $\per(\qgr R)$ does not have a full strong exceptional collection.
\end{enumerate}
\end{theorem}

Now we discuss tilting objects on the algebraic side.
Just as we were considering $\per(\qgr R)$ on the geometric side rather than $\DDD^{\bo}(\qgr R)$, we consider the subcategory $\CM_0^{\Z}R$ of $\CM^{\Z}R$ in \eqref{CM0}. This can be described as
\begin{equation}\label{CM02}
\CM_0^{\Z}R=\{X\in\CM^{\Z}R\mid K\otimes_RX\in\proj K\}
\end{equation}
(Proposition \ref{CM0=CM02}).
Moreover $\CM^{\Z}_0R= \CM^{\Z}R$ holds if and only if $R$ is reduced.

There exists an integer $a\in\Z$ such that $\Ext^1_R(k,R(a))\simeq k$ in $\mod^{\Z}R$. We call $a$ the \emph{$a$-invariant} ($-a$ the \emph{Gorenstein parameter}) of $R$ \cite{BH,GN}.
It can be characterized as the smallest integer $a$
such that $R_{>a}=K_{>a}$ (Lemma \ref{basic properties of K}(a)).
When $R$ has a non-negative $a$-invariant, $\underline{\CM}_0^{\Z}R$ always has a tilting object by the following result.

\begin{theorem}\label{general tilting}
Under the setting (R1) and (R2), assume moreover that the $a$-invariant $a$ of $R$ is non-negative.
Then the following holds true.
\begin{enumerate}[\rm(a)]
\item $\underline{\CM}_0^{\Z}R$ has a tilting object
\[V:=\bigoplus_{i=1}^{a+p}R(i)_{\ge0}= \bigoplus_{i=1}^{a+p} R_{\ge i}(i).\]
\item We have a triangle equivalence $\underline{\CM}_0^{\Z}R\simeq\KKK^{\bo}(\proj \Gamma)$ for $\Gamma:=\underline{\End}^{\Z}_R(V)$.
\item We have
{\small\begin{equation}\label{gamma}
\Gamma\simeq\End^{\Z}_R(V)=
\begin{bmatrix}
R_0&0&\cdots&0&0&
0&0&\cdots&0&0\\
R_1&R_0&\cdots&0&0&
0&0&\cdots&0&0\\
\vdots&\vdots&\ddots&\vdots&\vdots
&\vdots&\vdots&\cdots&\vdots&\vdots\\
R_{a-2}&R_{a-3}&\cdots&R_0&0&
0&0&\cdots&0&0\\
R_{a-1}&R_{a-2}&\cdots&R_1&R_0&
0&0&\cdots&0&0\\
K_a&K_{a-1}&\cdots&K_2&K_1&
K_0&K_{-1}&\cdots&K_{2-p}&K_{1-p}\\
K_{a+1}&K_a&\cdots&K_3&K_2&
K_{1}&K_0&\cdots&K_{3-p}&K_{2-p}\\
\vdots&\vdots&\vdots&\vdots&\vdots
&\vdots&\vdots&\ddots&\vdots&\vdots\\
K_{a+p-2}&K_{a+p-3}&\cdots&K_p&K_{p-1}&
K_{p-2}&K_{p-3}&\cdots&K_{0}&K_{-1}\\
K_{a+p-1}&K_{a+p-2}&\cdots&K_{p+1}&K_p&
K_{p-1}&K_{p-2}&\cdots&K_{1}&K_0
\end{bmatrix}.
\end{equation}}
\item $\Gamma$ is an Iwanaga-Gorenstein $k$-algebra, that is, $\injdim\Gamma_\Gamma=\injdim_{\Gamma}\Gamma<\infty$.
\item $R$ is reduced if and only if $\Gamma$ has finite global dimension.
\end{enumerate}
\end{theorem}

The above $V$ is an analog of the tilting object in $\underline{\mod}^{\Z}A$ given in \cite{Ya} for a $\Z$-graded finite dimensional self-injective algebras.

As a special case of Theorem \ref{general tilting}, we obtain the following result for reduced rings.

\begin{corollary}\label{general tilting reduced}
Under the setting (R1) and (R2), assume moreover that $R$ is reduced and not regular.
Then the following holds true.
\begin{enumerate}[\rm(a)]
\item The $a$-invariant $a$ of $R$ is non-negative.
\item $\underline{\CM}^{\Z}R$ has a tilting object
\[V:=\bigoplus_{i=1}^{a+p}R(i)_{\ge0}= \bigoplus_{i=1}^{a+p} R_{\ge i}(i).\]
\item We have a triangle equivalence $\underline{\CM}^{\Z}R\simeq\DDD^{\bo}(\mod\Gamma)$, where $\Gamma:=\underline{\End}^{\Z}_R(V)$ is a finite dimensional $k$-algebra with finite global dimension.
\item There exists an ordering in the isomorphism classes of indecomposable direct summands of $V$ which forms a full strong exceptional collection in $\underline{\CM}^{\Z}R$.
\end{enumerate}
\end{corollary}

Note that, for the case of hypersurfaces, a different tilting object with a much nicer endomorphism algebra was constructed in \cite{HI} before this paper.

Now we discuss the case when $R$ has a negative $a$-invariant. 
In this case the following result shows that $\underline{\CM}_0^{\Z}R$ never has a tilting object except for the trivial case, where we denote by $\thick P$ the smallest thick subcategory containing $P$.
We refer to Section \ref{negativa case} for a concrete example.

\begin{theorem}\label{negative case}
Under the setting (R1) and (R2), assume moreover that the $a$-invariant $a$ of $R$ is negative.
Then the following holds true.
\begin{enumerate}[\rm(a)]
\item $\underline{\CM}_0^{\Z}R$ has a silting object $\bigoplus_{i=1}^{a+p}R(i)_{\ge0}$.
\item We have a triangle equivalence $\underline{\CM}_0^{\Z}R\simeq\KKK^{\bo}(\proj \Lambda)/\thick P$, where $\Lambda$ is given by \eqref{lambda} and $P$ is the projective $\Lambda$-module corresponding to the first $-a$ rows.
\item $\underline{\CM}_0^{\Z}R$ has a tilting object if and only if $R$ is regular.
\end{enumerate}
\end{theorem}

As an application of our results, we calculate the Grothendieck groups of the triangulated categories $\per(\qgr R)$ and $\underline{\CM}_0^{\Z}R$.
We decompose $K$ into a product $K=K^1\times\cdots\times K^m$ of rings $K^i$ which are ring-indecomposable.
For each $1\le i\le m$, let $p_i$ be the smallest positive integer satisfying $K^i(p_i)\simeq K^i$ in $\mod^{\Z}K$.

\begin{corollary}\label{grothendieck}
Under the setting (R1) and (R2), the following holds true.
\begin{enumerate}[\rm(a)]
\item The Grothendieck group of $\per(\qgr R)$ is a free abelian group of rank $\sum_{i=1}^mp_i$.
\item The Grothendieck group of $\underline{\CM}_0^{\Z}R$ is a free abelian group of rank $a+\sum_{i=1}^mp_i$.
\end{enumerate}
\end{corollary}

Another application is the following observation, which shows that our category $\underline{\CM}_0^{\Z}R$ is a rich source of triangulated categories.

\begin{corollary}\label{abundance of CM}
Let $A$ be a $\Z$-graded commutative artinian Gorenstein ring such that $A=A_{\ge0}$ and $A_0$ is a field.
Then there exists a ring $R$ satisfying (R1) and (R2) such that $\underline{\CM}_0^{\Z}R$ is triangle equivalent to $\KKK^{\bo}(\proj^{\Z/a\Z}A)$, where $a$ is the $a$-invariant of $A$ and we regard $A$ as a $(\Z/a\Z)$-graded ring naturally.
\end{corollary}

\noindent
{\bf Conventions }
All modules are right modules. The composition of morphisms (respectively, arrows) $f\colon X\to Y$ and $g\colon Y\to Z$ is denoted by $gf$.
We denote by $k$ an arbitrary field.

\medskip\noindent
{\bf Acknowledgements }
The authors would like to thank Tokuji Araya, Martin Herschend and Ryo Takahashi for useful discussions at the first stage of this project.
They also thank Shiro Goto, Helmut Lenzing, Atsushi Takahashi, Yuji Yoshino for valuable discussions.
We thank the Centre de Recerca Matem\`atica, and the organizers of the Workshop ``(Re)emerging methods in Commutative Algebra and Representation Theory'' in February 2015 where a part of this collaboration was done. We thank the referee for valuable suggestions.

\section{Examples}

\subsection{Hypersurface singularities} 

In this subsection, we study hypersurface singularities in dimension one with standard grading. In the rest, let $k$ be an arbitrary field,
\[R=k[x,y]/(f)\ \mbox{ with }\ \deg x=\deg y=1,\ \mbox{ and }\ \Gamma=\underline{\End}^{\Z}_R(V)\]
for the tilting object $V$ given in Theorem \ref{general tilting}.
Then $a=n-2$ holds for $n:=\deg f$, and there is a triangle equivalence
\[\underline{\CM}^{\Z}_0R\simeq\KKK^{\bo}(\proj\Gamma).\]
We show that $\Gamma$ has selfinjective dimension at most $2$ and possibly infinite global dimension. More precisely, we prove the following results in Section \ref{subsection proof of HS}.

\begin{theorem}\label{example hypersurface}
Under the above setting, the following holds true.
\begin{enumerate}[\rm(a)]
\item $\Gamma$ is an Iwanaga-Gorenstein $k$-algebra with $\injdim\Gamma_\Gamma=\injdim_{\Gamma}\Gamma\le2$.
\item Assume $n\ge4$. Then there is no Iwanaga-Gorenstein $k$-algebra $\Gamma'$ which is derived equivalent to $\Gamma$ and satisfies $\injdim\Gamma'_{\Gamma'}=\injdim_{\Gamma'}\Gamma'\le1$.
\end{enumerate}
In the rest, we assume $f = \prod_{i=1}^mf_i^{n_i}$, where $f_i=\alpha_ix+\beta_iy$ is a linear form such that $(f_i)\neq (f_j)$ for all $i\neq j$, and $n_i$ is a positive integer.
\begin{enumerate}[\rm(c)]
\item[\rm(c)] Let $K^i$ be the $\Z$-graded total quotient ring of $R^i=k[x,y]/(f_{i}^{n_i})$ for $1\le i\le m$. Then $K_{\ge0}\simeq K^1_{\ge0}\times\cdots\times K^m_{\ge0}$ holds and $K^i_{\ge0}$ is indecomposable in $\CM^{\Z}R$.
\item[\rm(d)] Let $(\alpha_i':\beta_i')\in\P^1_k$ be a point different from $(\alpha_i:\beta_i)$. Then $\Gamma$ is presented by the quiver
{\small
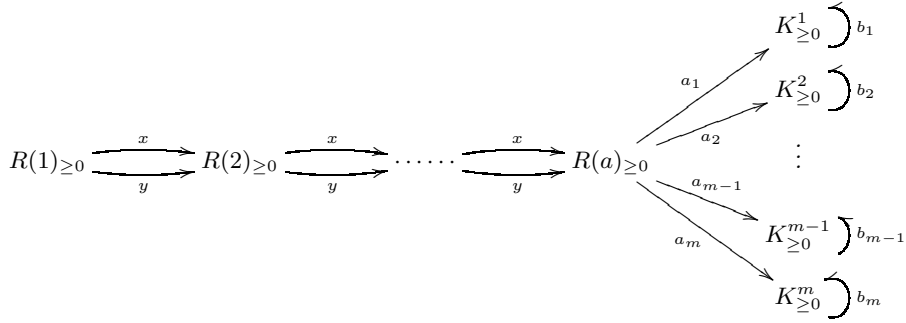
\[\xymatrix@C=4em@R=.8em{
&&&&K^1_{\ge0}\ar@(dr,ur)_{b_1}\\
&&&&K^2_{\ge0} \ar@(dr,ur)_{b_2}\\
R(1)_{\ge0} \ar@/^0.4pc/[r]^{x} \ar@/_0.4pc/[r]_{y} &
R(2)_{\ge0} \ar@/^0.4pc/[r]^{x} \ar@/_0.4pc/[r]_{y} &  
\cdots \cdots  \ar@/^0.4pc/[r]^{x} \ar@/_0.4pc/[r]_{y} &
R(a)_{\ge0} \ar[ruu]^{a_1}  \ar[ru]_{a_2}  \ar[rd]^{a_{m-1}}  \ar[rdd]_{a_m}  &
\vdots & \\
&&&&K^{m-1}_{\ge0}\ar@(dr,ur)_{b_{m-1}}\\
&&&&K^m_{\ge0}\ar@(dr,ur)_{b_m}}
\]}
with relations
\[xy=yx,\ b_i^{n_i}=0,\ a_i(\alpha_i x + \beta_i y)=b_ia_i(\alpha'_ix+ \beta'_i y).\]
\item[\rm(e)] $n_1=\cdots=n_m=1$ holds if and only if $\gl\Gamma<\infty$ if and only if $\gl\Gamma\le2$.  
\end{enumerate}
\end{theorem}

In (e), one can show that $\Gamma$ is derived equivalent to $k\times k$ if $n=2$, a path algebra of type $D_4$ if $n=3$, and a canonical algebra of type $(2,2,2,2)$ of $n=4$ (see \cite{HI} and Proposition \ref{Tpq singularity}(a) below).
Also notice that, if $n\ge4$, then $\Gamma$ is not derived equivalent to a hereditary $k$-algebra by (b) above.

\subsection{Simple curve singularities}
In this subsection we study simple curve singularities.
They are precisely the ADE singularities when the base field is algebraically closed and the characteristic is different from $2$, $3$ and $5$ \cite[Section 9]{LW}.
Our result is the following.

\begin{theorem}\label{simple singularity}
Let $R=k[x,y]/(f)$ be an ADE singularity over an arbitrary field $k$ with minimal grading given by the list below.
Then $\underline{\CM}^{\Z}R$ is triangle equivalent to $\DDD^{\bo}(\mod kQ)$, where $Q$ is a Dynkin quiver of the following type.
\[\begin{array}{|c||c|c|c|c|c|}\hline
R&A_{n}&D_{n}&E_6&E_7&E_8\\ \hline
f&x^{n+1}-y^2&x^{n-1}-xy^2&x^4-y^3&x^3y-y^3&x^5-y^3\\ \hline
(\deg x,\deg y)&\begin{array}{cc}(1,\frac{n+1}{2})&\mbox{$n$ is odd}\\ (2,n+1)&\mbox{$n$ is even}\end{array}&\begin{array}{cc}(2,n-2)&\mbox{$n$ is odd}\\ (1,\frac{n}{2}-1)&\mbox{$n$ is even}\end{array}&(3,4)&(2,3)&(3,5)\\ \hline
Q&\begin{array}{cc}D_{\frac{n+3}{2}}&\mbox{$n$ is odd}\\ A_{n}&\mbox{$n$ is even}\end{array}
&\begin{array}{cc}A_{2n-3}&\mbox{$n$ is odd}\\ D_{n}&\mbox{$n$ is even}\end{array}&E_6&E_7&E_8\\ \hline
\end{array}\]
\end{theorem}

This is an analogue of \eqref{simple dynkin} in dimension $2$.
The difference of types of $R$ and $Q$ was observed in \cite{DW} (see also \cite{Y,LW}).
We will prove Theorem \ref{simple singularity} in Section \ref{subsection proof of simple singularity}.

Our Theorem \ref{simple singularity} immediately recovers the following well known results. 

\begin{corollary}
Let $R=k[x,y]/(f)$ be as in Theorem \ref{simple singularity}, and $\widehat{R}$ the completion of $R$ at $R_{>0}$.
\begin{enumerate}[\rm(a)]
\item \cite{A} There are only finitely many indecomposable objects in $\CM^{\Z}R$
up to isomorphisms and degree shift. The stable Auslander-Reiten quiver of $\CM^{\Z}R$ is $\Z Q$ (see \cite{Ha1}).
\item \cite{J,DR,GK} There are only finitely many indecomposable objects in $\CM\widehat{R}$ up to isomorphisms.
\end{enumerate}
\end{corollary}

\begin{proof}
(a) is immediate from Theorem \ref{simple singularity}. 
(b) follows from (a) and \cite[Theorem 5]{AR2}.
\end{proof}

\subsection{Curve singularities $T_{pq}$}

Drozd-Greuel classified commutative noetherian rings in dimension one which are CM-tame in terms of $T_{pq}$ singularities \cite{DG}.
Recall that \emph{$T_{pq}$ singularities} over an algebraically closed field $k$ whose characteristic is different from $2$ have the form $k[x,y]/(f)$, where $f=x^p+\lambda x^2y^2+y^q$ with $p\le q$ and $\lambda\in k\backslash\{0,1\}$.

In this subsection, we deal with $\Z$-graded $T_{pq}$ singularities such that the variables $x$ and $y$ are homogeneous. This is precisely the case when $(p,q,\deg x,\deg y)$ is either $(4,4,1,1)$ or $(3,6,2,1)$. Our result below covers a slightly more general class of rings.
Recall that a \emph{canonical algebra of type $(2,2,2,2)$} is given by the following quiver with relations for $\lambda\in k\backslash\{0,1\}$ \cite{R}.
\[{\small\xymatrix@C=4em@R=.1em{
 & \bullet \ar[rdd]^{b_1}& \\
 & \bullet \ar[rd]|{b_2} &&b_1a_1+b_2a_2+b_3a_3=0\\
 \bullet \ar[ruu]^{a_1} \ar[ru]|{a_2} \ar[rd]|{a_3} \ar[rdd]_{a_4} & & \bullet &b_1a_1+\lambda b_2a_2+b_4a_4=0.\\
 & \bullet \ar[ru]|{b_3}&\\
  & \bullet \ar[ruu]_{b_4}&
}}\]
This algebra is derived equivalent to the weighted projective line of type $(2,2,2,2)$ \cite{GL1}.

We will prove the following result in Section \ref{proof for Tpq}.

\begin{proposition}\label{Tpq singularity}
Let $k$ be an arbitrary field and $R=k[x,y]/(f)$, where
\begin{enumerate}[\rm(a)]
\item $f=\prod_{i=1}^4(x-\alpha_iy)$ and $(\deg x,\deg y)=(1,1)$, or
\item $f=\prod_{i=1}^3(x-\alpha_iy^2)$ and $(\deg x,\deg y)=(2,1)$.
\end{enumerate}
If $R$ is reduced, then $\underline{\CM}^{\Z}R$ is triangle equivalent to $\DDD^{\bo}(\mod C)$, where $C$ is a canonical algebra of type $(2,2,2,2)$ with $\lambda=(\alpha_1-\alpha_4)(\alpha_2-\alpha_3)(\alpha_1-\alpha_3)^{-1}(\alpha_2-\alpha_4)^{-1}$ for (a) and 
$\lambda=(\alpha_2-\alpha_3)(\alpha_1-\alpha_3)^{-1}$ for (b).
\end{proposition}

Consequently, $\underline{\CM}^{\Z}R$ is triangle equivalent to $\DDD^{\bo}(\coh\X)$, where $\X$ is the weighted projective line of type $(2,2,2,2)$. It will be interesting to find out a direct explanation of this equivalence.

\subsection{Non-reduced examples}\label{negativa case}

In this subsection, let $k$ be an arbitrary field and
\[R=k[x,y]/(y^2)\ \mbox{ with $\deg x=n\ge1$ and $\deg y=1$.}\]
When $n\ge2$, this gives a typical example of rings with negative $a$-invariant.
It is known as a Bass order in a non-semisimple algebra \cite{HN}
and as a CM-countable ring \cite{BGS,LW}.

\begin{proposition}\label{no tilting}
Under the above setting, the following holds true.
\begin{enumerate}[\rm(a)]
\item The $a$-invariant of $R$ is $1-n$, and we have $K(n)\simeq K$.
\item $\per(\qgr R)$ is triangle equivalent to $\KKK^{\bo}(\proj\Lambda)$ for $\Lambda:=kQ/(z^2)$, where $Q$ is the following quiver with $Q_0=\Z/n\Z$.
{\small
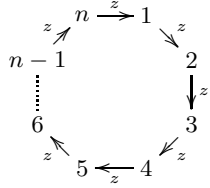
\[
\begin{xy}
0;<2.4pt,0pt>:<0pt,2.4pt>:: 
(0,0) *+{n}="a",
(10,0) *+{1}="b",
(17,-7) *+{2}="c",
(17,-17) *+{3}="d",
(10,-24) *+{4}="e",
(0,-24) *+{5}="f",
(-7,-17) *+{6}="g",
(-7,-7) *+{n-1}="h",
\ar "a";"b"^{z}
\ar "b";"c"^{z}
\ar "c";"d"^{z}
\ar "d";"e"^{z}
\ar "e";"f"^{z}
\ar "f";"g"^{z}
\ar@{.} "g";"h"
\ar "h";"a"^{z}
\end{xy}
\]}
\item 
The Auslander-Reiten quiver of $\per(\qgr R)\simeq\KKK^{\bo}(\proj\Lambda)$ has $n$ connected components.
For $i\in\Z/n\Z$, let $P^{i}=e_i\Lambda$ for the idempotent $e_i\in\Lambda$ corresponding to the vertex $i$, and for $a,b\in\Z$ with $b\ge0$, let $X^{i}_{a,b}$ be the complex
\[
\cdots\to0\to P^i\xrightarrow{z} P^{i+1}\xrightarrow{z} P^{i+2}\xrightarrow{z}\cdots\xrightarrow{z}P^{i+b-1}\xrightarrow{z} P^{i+b}\to0\to\cdots
\]
whose non-zero degrees are $a,a+1,\ldots,a+b$.
Then the following is a connected component for $i\in\Z/n\Z$.
{\small
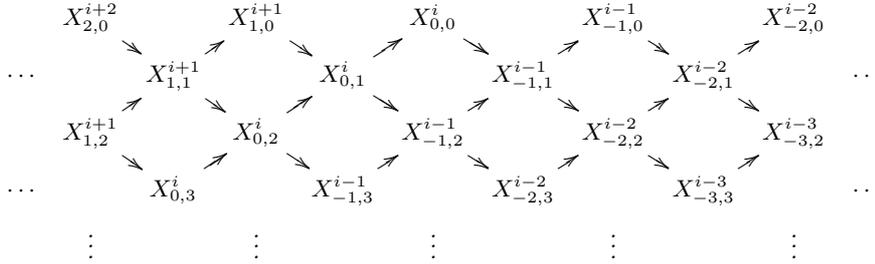
\[\xymatrix@R.5em@C.5em{
&X^{i+2}_{2,0}\ar[rd]&&X^{i+1}_{1,0}\ar[rd]&&X^i_{0,0}\ar[rd]&&X^{i-1}_{-1,0}\ar[rd]&&X^{i-2}_{-2,0}\\
\cdots&&X^{i+1}_{1,1}\ar[rd]\ar[ru]&&X^{i}_{0,1}\ar[ru]\ar[rd]&&X^{i-1}_{-1,1}\ar[ru]\ar[rd]&&X^{i-2}_{-2,1}\ar[ru]\ar[rd]&&\cdots\\
&X^{i+1}_{1,2}\ar[ru]\ar[rd]&&X^{i}_{0,2}\ar[ru]\ar[rd]&&X^{i-1}_{-1,2}\ar[ru]\ar[rd]&&X^{i-2}_{-2,2}\ar[ru]\ar[rd]&&X^{i-3}_{-3,2}\\
\cdots&&X^{i}_{0,3}\ar[ru]&&X^{i-1}_{-1,3}\ar[ru]&&X^{i-2}_{-2,3}\ar[ru]&&X^{i-3}_{-3,3}\ar[ru]&&\cdots\\
&\vdots&&\vdots&&\vdots&&\vdots&&\vdots
}\]}
\item 
$\underline{\CM}^{\Z}_0R$ is triangle equivalent to $\KKK^{\bo}(\proj\Lambda)/\thick P$ 
where $P=\bigoplus_{i=1}^{n-1}P^i$.
\item 
$\underline{\CM}_0^{\Z}R$ has a silting object $R(1)_{\ge0}$, and has a tilting object if and only if $n=1$.
It is triangle equivalent to the perfect derived category $\per k[w]/(w^2)$ for the DG (=differential graded) algebra $k[w]/(w^2)$ with $\deg w=1-n$ and zero differential.
\item The Auslander-Reiten quiver of $\underline{\CM}^{\Z}_0R$ has $n$ connected components. 
For $i>0$ and $j\in\Z$, let $R^i:=R+\langle x^{-\ell}y\mid 1\le\ell\le i\rangle_k$ and $R^{i,j}=R^i(j)$.
Then the following is a connected component for $j\in\Z/n\Z$.
{\small
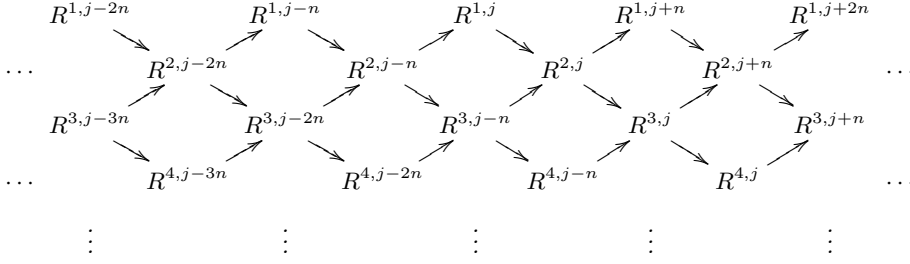
\[\xymatrix@R.8em@C0em{
&R^{1,j-2n}\ar[rd]&&R^{1,j-n}\ar[rd]&&R^{1,j}\ar[rd]&&R^{1,j+n}\ar[rd]&&R^{1,j+2n}\\
\cdots&&R^{2,j-2n}\ar[rd]\ar[ru]&&R^{2,j-n}\ar[ru]\ar[rd]&&R^{2,j}\ar[ru]\ar[rd]&&R^{2,j+n}\ar[ru]\ar[rd]&&\cdots\\
&R^{3,j-3n}\ar[ru]\ar[rd]&&R^{3,j-2n}\ar[ru]\ar[rd]&&R^{3,j-n}\ar[ru]\ar[rd]&&R^{3,j}\ar[ru]\ar[rd]&&R^{3,j+n}\\
\cdots&&R^{4,j-3n}\ar[ru]&&R^{4,j-2n}\ar[ru]&&R^{4,j-n}\ar[ru]&&R^{4,j}\ar[ru]&&\cdots\\
&\vdots&&\vdots&&\vdots&&\vdots&&\vdots
}\]}
\end{enumerate}
\end{proposition}

We will give a proof of Proposition \ref{no tilting} in Section \ref{subsection proof of no tilting}.
Note that the Auslander-Reiten quivers of $\per(\qgr R)$ and $\underline{\CM}^{\Z}_0R$ are isomorphic, but they are not triangle equivalent.

\section{Realizing Verdier quotients as thick subcategories}

Throughout this subsection, we assume that $A$ is a \emph{$\Z$-graded Iwanaga-Gorenstein ring}, that is,
\begin{itemize}
\item $A$ is a noetherian ring on each side with $\injdim A_A<\infty$ and $\injdim_AA<\infty$.
\end{itemize}
We denote by $\mod^{\Z}A$ the category of $\Z$-graded finitely generated (right) $A$-modules,
by $\proj^{\Z}A$ the category of $\Z$-graded finitely generated projective $A$-modules, and by $\mod_0^{\Z}A$ the category of $\Z$-graded $A$-modules of finite length.

Under certain conditions, it is known \cite{O,IY2} that two Verdier quotients $\DDD^{\bo}(\mod^{\Z}A)/\KKK^{\bo}(\proj^{\Z}A)$ and $\DDD^{\bo}(\mod^{\Z}A)/\DDD^{\bo}(\mod^{\Z}_0A)$ can be realized as thick subcategories of $\DDD^{\bo}(\mod^{\Z}A)$.
The aim of this section is to give an analogous result for the thick subcategory
\[\DD_A:=\thick\{\proj^{\Z}A,\mod^{\Z}_0A\}\subseteq\DDD^{\bo}(\mod^{\Z}A),\]
and its Verdier quotients
$\DD_A/\KKK^{\bo}(\proj^{\Z}A)$ and $\DD_A/\DDD^{\bo}(\mod^{\Z}_0A)$.

For a subset $I$ of $\Z$, let $\mod^IA$ be the full subcategory of $\mod^{\Z}A$ consisting of all $X$ satisfying $X_i=0$ for all $i\in\Z\setminus I$.
For an integer $\ell\in\Z$, let $\mod^{\ge\ell}A:=\mod^{[\ell,\infty)}A$,
$\mod^{\le\ell}A:=\mod^{(-\infty,\ell]}A$ and so on.
Then $\DDD^{\bo}(\mod^{\ge\ell}A)$ can be regarded as a thick subcategory of $\DDD^{\bo}(\mod^{\Z}A)$. Let
\[\DD_A^{\ge\ell}=\DD_A^{>\ell-1}:=\DD_A\cap\DDD^{\bo}(\mod^{\ge\ell}A).\]
Let $\mod^{\ge\ell}_0A:=\mod^{\ge\ell}A\cap\mod^{\Z}_0A$, $\mod^{\le\ell}_0A:=\mod^{\le\ell}A\cap\mod^{\Z}_0A=\mod^{\le\ell}A$ and so on.
Similarly, let $\proj^IA$ be the full subcategory of $\proj^{\Z}A$ consisting of all $P$ which are generated by homogeneous elements of degrees in $I$.
Let $\proj^{\ge\ell}A:=\proj^{[\ell,\infty)}A$, $\proj^{\le\ell}A:=\proj^{(-\infty,\ell]}A$ and so on.

Since $A$ is Iwanaga-Gorenstein, we have a duality \cite[Corollary 2.11]{Miy}
\[(-)^*:=\RHom_A(-,A):\DDD^{\bo}(\mod^{\Z}A)\simeq\DDD^{\bo}(\mod^{\Z}A^{\op}).\]
We consider the following three conditions.
\begin{itemize}
\item[(A1)] $A=\bigoplus_{i\ge0}A_i$ and $\gl A_0<\infty$.
\item[(A2)] $A_0$ is an artinian ring.
\item[(A3)] There exists $a\in\Z$ such that $(-)^*$ restricts to a duality $(-)^*:\DDD^{\bo}(\mod^0A)\simeq\DDD^{\bo}(\mod^aA^{\op})$.
\end{itemize}
For example, our $R$ satisfying (R1) and (R2) satisfies these conditions for the $a$-invariant $a$ of $R$.

Let $\XX$ and $\YY$ be full subcategories in a triangulated category $\TT$.
We denote by $\XX*\YY$ the full subcategory of $\TT$ whose objects consisting of $Z\in\TT$ such that there is a triangle $X\to Z\to Y\to X[1]$ with $X\in\XX$ and $Y\in\YY$.
When $\Hom_{\TT}(\XX,\YY)=0$ holds, we write $\XX*\YY=\XX\perp\YY$. For full subcategories $\XX_1,\ldots,\XX_n$, we define $\XX_1*\cdots*\XX_n$ and $\XX_1\perp\cdots\perp\XX_n$ inductively.

We are ready to state the following main result in this section.

\begin{theorem}\label{embedding2}
Let $A$ be a $\Z$-graded Iwanaga-Gorenstein ring satisfying $A=A_{\ge0}$, and $\ell$ an integer.
\begin{enumerate}[\rm(a)]
\item If the condition (A1) is satisfied, then we have a semi-orthogonal decomposition
\begin{eqnarray*}
\DD_A=\KKK^{\bo}(\proj^{<\ell}A)\perp(\DD_A^{\ge\ell}\cap(\DD_{A^{\op}}^{>-\ell})^*)\perp\KKK^{\bo}(\proj^{\ge\ell}A).
\end{eqnarray*}
The natural functor $\DD_A\to\DD_A/\KKK^{\bo}(\proj^{\Z}A)$ restricts to a triangle equivalence
\begin{eqnarray*}
\DD^{\ge\ell}_A\cap(\DD^{>-\ell}_{A^{\op}})^*\simeq\DD_A/\KKK^{\bo}(\proj^{\Z}A).
\end{eqnarray*}
\item If the conditions (A2) and (A3) are satisfied, then we have a semi-orthogonal decomposition
\begin{eqnarray*}
\DD_A=\DDD^{\bo}(\mod_0^{\ge\ell}A)\perp(\DD_A^{\ge\ell}\cap(\DD_{A^{\op}}^{>a-\ell})^*)\perp\DDD^{\bo}(\mod_0^{<\ell}A).
\end{eqnarray*}
The natural functor $\DD_A\to\DD_A/\DDD^{\bo}(\mod^{\Z}_0A)$ restricts to a triangle equivalence
\begin{eqnarray*}
\DD^{\ge\ell}_A\cap(\DD^{>a-\ell}_{A^{\op}})^*\simeq\DD_A/\DDD^{\bo}(\mod^{\Z}_0A).
\end{eqnarray*}
\item Assume that the conditions (A1), (A2) and (A3) are satisfied. If $a\ge0$, then we have a semi-orthogonal decomposition
\[\DD_A^{\ge\ell-a}\cap(\DD_{A^{\op}}^{>a-\ell})^*=(\DD_A^{\ge\ell}\cap(\DD_{A^{\op}}^{>a-\ell})^*)\perp\DDD^{\bo}(\mod^{[\ell-a,\ell-1]}A).\]
If $a\le0$, then we have a semi-orthogonal decomposition
\[\DD_A^{\ge\ell}\cap(\DD_{A^{\op}}^{>a-\ell})^*=\KKK^{\bo}(\proj^{[\ell,\ell-a-1]}A)\perp(\DD_A^{\ge\ell-a}\cap(\DD_{A^{\op}}^{>a-\ell})^*).\]
\end{enumerate}
\end{theorem}

Immediately, we obtain the following analogue of Orlov's semi-orthogonal decompositions \cite{O}.

\begin{corollary}\label{SOD}
Assume that (A1), (A2) and (A3) are satisfied. For $\ell\in\Z$, there exist fully faithful triangle functors $F_\ell:\DD_A/\KKK^{\bo}(\proj^{\Z}A)\to\DD_A$ and $G_\ell:\DD_A/\DDD^{\bo}(\mod^{\Z}_0A)\to\DD_A$ and a semi-orthogonal decomposition
\[\begin{array}{ll}
G_\ell(\DD_A/\DDD^{\bo}(\mod^{\Z}_0A))\simeq\KKK^{\bo}(\proj^{[\ell,\ell-a-1]}A)\perp F_\ell(\DD_A/\KKK^{\bo}(\proj^{\Z}A))&\mbox{if $a<0$,}\\
F_\ell(\DD_A/\KKK^{\bo}(\proj^{\Z}A))\simeq G_\ell(\DD_A/\DDD^{\bo}(\mod^{\Z}_0A))&\mbox{if $a=0$,}\\
F_\ell(\DD_A/\KKK^{\bo}(\proj^{\Z}A))\simeq G_\ell(\DD_A/\DDD^{\bo}(\mod^{\Z}_0A))\perp\DDD^{\bo}(\mod^{[\ell-a,\ell-1]}A)&\mbox{if $a>0$.}
\end{array}\]
\end{corollary}

We refer to \cite{MY} for analogous results to Theorem \ref{embedding2}.

The rest of this section is devoted to prove Theorem \ref{embedding2} and Corollary \ref{SOD}. 
We start with the following easy observation.

\begin{lemma}\label{modular}
Let $\TT$ be a triangulated category and $\TT=\XX\perp\YY$ a semi-orthogonal decomposition.
\begin{enumerate}[\rm(a)]
\item If $\ZZ$ is a thick subcategory of $\TT$ such that $\XX\subseteq\ZZ$, then $\ZZ=\XX\perp(\YY\cap\ZZ)$.
\item If $\ZZ$ is a thick subcategory of $\TT$ such that $\YY\subseteq\ZZ$, then $\ZZ=(\XX\cap\ZZ)\perp\YY$.
\item If $\TT=\XX'\perp\YY'$ is a semi-orthogonal decomposition such that $\XX\subseteq\XX'$ (or equivalently, $\YY\supseteq\YY'$), then we have a semi-orthogonal decomposition $\TT=\XX\perp(\YY\cap\XX')\perp\YY'$.
\end{enumerate}
\end{lemma}

\begin{proof}
(a) and (b) are easy. By (a), we have $\XX'=\XX\perp(\YY\cap\XX')$ and hence $\TT=\XX'\perp\YY'=\XX\perp(\YY\cap\XX')\perp\YY'$.
\end{proof}

We need the following elementary observation (e.g.\ \cite[2.3]{O}).

\begin{proposition}\label{SOD of K}
Let $A$ be a $\Z$-graded Iwanaga-Gorenstein ring satisfying (A1). 
Then there exists a semi-orthogonal decomposition
$\KKK(\proj^{\Z}A)=\KKK(\proj^{<\ell}A)\perp\KKK(\proj^{\ge\ell}A)$
which gives
$\KKK^{\bo}(\proj^{\Z}A)=\KKK^{\bo}(\proj^{<\ell}A)\perp\KKK^{\bo}(\proj^{\ge\ell}A)$ and $\DDD^{\bo}(\mod^{\Z}A)=\KKK^{\bo}(\proj^{<\ell}A)\perp\DDD^{\bo}(\mod^{\ge\ell}A)$.
\end{proposition}

Now we prove Theorem \ref{embedding2}(a).

\begin{proof}[Proof of Theorem \ref{embedding2}(a)]
We have $\DDD^{\bo}(\mod^{\Z}A)=\KKK^{\bo}(\proj^{<\ell}A)\perp\DDD^{\bo}(\mod^{\ge\ell}A)$ by Proposition \ref{SOD of K}.
Applying Lemma \ref{modular}(a) to $\XX:=\KKK^{\bo}(\proj^{<\ell}A)\subseteq\ZZ:=\DD_A$, we have
\begin{equation}\label{first SOD}
\DD_A=\KKK^{\bo}(\proj^{<\ell}A)\perp(\DD_A\cap\DDD^{\bo}(\mod^{\ge\ell}A))=\KKK^{\bo}(\proj^{<\ell}A)\perp\DD_A^{\ge\ell}.
\end{equation}
Replacing $\ell$ by $1-\ell$, we have $\DD_{A^{\op}}=\KKK^{\bo}(\proj^{\le-\ell}A^{\op})\perp\DD_{A^{\op}}^{>-\ell}$. Applying $(-)^*$, we have
\begin{equation}\label{second SOD}
\DD_A=(\DD_{A^{\op}}^{>-\ell})^*\perp\KKK^{\bo}(\proj^{\le-\ell}A^{\op})^*=
(\DD_{A^{\op}}^{>-\ell})^*\perp\KKK^{\bo}(\proj^{\ge\ell}A).
\end{equation}
Since $\DD_A^{\ge\ell}\supseteq\KKK^{\bo}(\proj^{\ge\ell}A)$, applying Lemma \ref{modular}(c) to \eqref{first SOD} and \eqref{second SOD} gives
$\DD_A=\KKK^{\bo}(\proj^{<\ell}A)\perp
(\DD_A^{\ge\ell}\cap(\DD_{A^{\op}}^{>-\ell})^*)\perp
\KKK^{\bo}(\proj^{\ge\ell}A)$ as desired. 
The last assertion follows from
\[\DD_A/\KKK^{\bo}(\proj^{\Z}A)\stackrel{{\rm Prop. }\ref{SOD of K}}{=}\DD_A/(\KKK^{\bo}(\proj^{<\ell}A)\perp\KKK^{\bo}(\proj^{\ge\ell}A))
\stackrel{}{\simeq}\DD_A^{\ge\ell}\cap(\DD_{A^{\op}}^{>-\ell})^*.\qedhere\]
\end{proof}

We also need the following elementary observation (e.g.\ \cite[2.3]{O}).

\begin{proposition}\label{SOD of S}
Let $A$ be a $\Z$-graded Iwanaga-Gorenstein ring satisfying $A=A_{\ge0}$ and (A2). Then there exist semi-orthogonal decompositions
$\DDD^{\bo}(\mod_0^{\Z}A)=\DDD^{\bo}(\mod_0^{\ge\ell}A)\perp\DDD^{\bo}(\mod_0^{<\ell}A)$ and $\DDD^{\bo}(\mod^{\Z}A)=\DDD^{\bo}(\mod^{\ge\ell}A)\perp\DDD^{\bo}(\mod_0^{<\ell}A)$.
\end{proposition}

Now we prove Theorem \ref{embedding2}(b) and (c).

\begin{proof}[Proof of Theorem \ref{embedding2}(b)(c)]
We have $\DDD^{\bo}(\mod^{\Z}A)=\DDD^{\bo}(\mod^{\ge\ell}A)\perp\DDD^{\bo}(\mod_0^{<\ell}A)$ by Proposition \ref{SOD of S}.
Applying Lemma \ref{modular}(b) to $\YY:=\DDD^{\bo}(\mod_0^{<\ell}A)\subseteq\ZZ:=\DD_A$, we have 
\begin{equation}\label{first SOD2}
\DD_A=(\DD_A\cap\DDD^{\bo}(\mod^{\ge\ell}A))\perp\DDD^{\bo}(\mod_0^{<\ell}A)=\DD_A^{\ge\ell}\perp\DDD^{\bo}(\mod_0^{<\ell}A).
\end{equation}
Replacing $\ell$ by $a-\ell+1$, we have $\DD_{A^{\op}}=\DD_{A^{\op}}^{>a-\ell}\perp\DDD^{\bo}(\mod_0^{\le a-\ell}A^{\op})$. Applying $(-)^*$, we have
\begin{equation}\label{second SOD2}
\DD_A=\DDD^{\bo}(\mod_0^{\le a-\ell}A^{\op})^*\perp(\DD_{A^{\op}}^{>a-\ell})^*\stackrel{{\rm (A3)}}{=}\DDD^{\bo}(\mod_0^{\ge\ell}A)\perp(\DD_{A^{\op}}^{>a-\ell})^*.
\end{equation}
Since $\DD_A^{\ge\ell}\supseteq\DDD^{\bo}(\mod_0^{\ge\ell}A)$, applying Lemma \ref{modular}(c) to \eqref{first SOD2} and \eqref{second SOD2} gives
$\DD_A=\DDD^{\bo}(\mod_0^{\ge\ell}A)\perp(\DD_A^{\ge\ell}\cap(\DD_{A^{\op}}^{>a-\ell})^*)\perp\DDD^{\bo}(\mod_0^{<\ell}A)$ as desired.
The last assertion follows from
\[\DD_A/\DDD^{\bo}(\mod_0^{\Z}A)
\stackrel{{\rm Prop. }\ref{SOD of S}}{=}\DD_A/(\DDD^{\bo}(\mod_0^{\ge\ell}A)\perp\DDD^{\bo}(\mod_0^{<\ell}A))\stackrel{}{\simeq}\DD_A^{\ge \ell}\cap(\DD_{A^{\op}}^{>a-\ell})^*.\]

(c) Assume $a\ge0$. Then $\DD_{A}^{\ge\ell-a}=\DD_{A}^{\ge\ell}\perp\DDD^{\bo}(\mod^{[\ell-a,\ell-1]}A)$ holds.
Since $\DDD^{\bo}(\mod^{[\ell-a,\ell-1]}A)=\DDD^{\bo}(\mod^{[a+1-\ell,2a-\ell]}A^{\op})^*\subseteq(\DD_{A^{\op}}^{>a-\ell})^*$, the assertion follows from Lemma \ref{modular}(b).

Assume $a\le0$. Then $\DD_A^{\ge\ell}=\KKK^{\bo}(\proj^{[\ell,\ell-a-1]}A)\perp\DD_A^{\ge\ell-a}$ holds.
Since $\KKK^{\bo}(\proj^{[\ell,\ell-a-1]}A)=\KKK^{\bo}(\proj^{[a+1-\ell,-\ell]}A^{\op})^*\subseteq(\DD_{A^{\op}}^{>a-\ell})^*$, the assertion follows from Lemma \ref{modular}(a).
\end{proof}

\section{Proof of our results}

\subsection{Preliminaries}

We start with recalling the central notion of tilting objects.

\begin{definition}\label{define tilting}
Let $\TT$ be a triangulated category with suspension functor $[1]$.
A full subcategory of $\TT$ is \emph{thick} if it is closed under cones, $[\pm1]$ and direct summands.
We call an object $T\in\TT$ \emph{tilting} (respectively, \emph{silting}) if $\Hom_{\TT}(T,T[i])=0$ holds
for all integers $i\neq0$ (respectively, $i>0$), and the smallest thick subcategory of $\TT$ containing $T$ is $\TT$.
\end{definition}

The principal example of tilting objects appear in $\KKK^{\bo}(\proj\Lambda)$ for a ring $\Lambda$. It has a tilting object given by the stalk complex $\Lambda$ concentrated in degree zero, and a certain converse holds in the sense of Proposition \ref{tilting functor} below. More generally, tilting objects in $\KKK^{\bo}(\proj\Lambda)$ are precisely tilting complexes \cite{Ric} of $\Lambda$, where tilting $\Lambda$-modules \cite{Ha1} are special examples. Note that there are some variations of the definitions of tilting objects \cite{AHK,KK}, for example, the last condition $\thick T=\TT$ is sometimes replaced by  the condition ``if $X\in\TT$ satisfies $\Hom_{\TT}(T,X[i])=0$ for all $i\in\Z$, then $X=0$''. If $\TT$ has arbitrary coproducts, then $T$ is assumed to be compact and ``thick subcategory'' in the last condition is replaced by ``localizing subcategory''.

Recall that a triangulated category is called \emph{algebraic} if it is triangle equivalent to the stable category of a Frobenius category.
Let us recall the following well known result due to Keller \cite{Ke} (see \cite{Ki2} for a detailed proof).

\begin{proposition}\cite{Ke}\label{tilting functor}
Let $\TT$ be an algebraic triangulated category with a tilting object $T$. 
There exists a triangle equivalence $F:\TT\to\KKK^{\bo}(\proj\End_{\TT}(T))$ up to direct summands such that $F(T)=\End_{\TT}(T)$. It is a triangle equivalence if $\TT$ is idempotent complete.
\end{proposition}

Let $k$ be a field and $D$ the $k$-dual.
For a finite dimensional $k$-algebra $\Lambda$, we denote by
\[\nu=-\Lotimes_\Lambda(D\Lambda):\KKK^{\bo}(\proj\Lambda)\simeq\KKK^{\bo}(\inj\Lambda)\]
the \emph{Nakayama functor}.
If $\Lambda$ is Iwanaga-Gorenstein, then $\nu$ is an autoequivalence of $\KKK^{\bo}(\proj\Lambda)=\KKK^{\bo}(\inj\Lambda)$.
The following result due to Happel is also well known.

\begin{proposition}\cite{Ha1}\label{AR duality for D}
Let $\Lambda$ be a finite dimensional $k$-algebra. Then we have a functorial isomorphism
\[D\Hom_{\DDD(\Mod\Lambda)}(X,Y)\simeq\Hom_{\DDD(\Mod\Lambda)}(Y,\nu X)\]
for any $X\in\KKK^{\bo}(\proj\Lambda)$ and $Y\in\DDD(\Mod\Lambda)$.
In particular, if $\Lambda$ is Iwanaga-Gorenstein, then $\KKK^{\bo}(\proj\Lambda)$ has a Serre functor $\nu$.
\end{proposition}

Now we prove the following general observation.

\begin{proposition}\label{IwanagaGorenstein}
Let $\TT$ be a Hom-finite $k$-linear algebraic triangulated category with Serre functor $\S$.
Let $T\in\TT$ be a tilting object and $\Lambda:=\End_{\TT}(T)$.
Then the following holds true.
\begin{enumerate}[\rm(a)]
\item $\Lambda$ is an Iwanaga-Gorenstein $k$-algebra.
\item There is a triangle equivalence $F:\TT\simeq\KKK^{\bo}(\proj\Lambda)$ up to direct summands and the following commutative diagram up to an isomorphism of functors.
\[\xymatrix@R1.5em{
\TT\ar[rr]^F\ar[d]^{\S}&&\KKK^{\bo}(\proj\Lambda)\ar[d]^{\nu}\\
\TT\ar[rr]^F&&\KKK^{\bo}(\proj\Lambda)
}\]
\end{enumerate}
\end{proposition}

\begin{proof}
(a) By Proposition \ref{tilting functor}, we may regard $\TT$ as a full triangulated subcategory of $\KKK^{\bo}(\proj\Lambda)$ and $T=\Lambda$. Then we have isomorphisms
\[\Hom_{\TT}(-,\S\Lambda)\simeq D\Hom_{\TT}(\Lambda,-)=D\Hom_{\DDD(\Mod\Lambda)}(\Lambda,-)|_{\TT}\simeq
\Hom_{\DDD(\Mod\Lambda)}(-,D\Lambda)|_{\TT}.\]
By Yoneda's Lemma, there is a morphism $f:\S\Lambda\to D\Lambda$ in $\DDD(\Mod\Lambda)$ which induces an isomorphism $\Hom_{\TT}(-,\S \Lambda)\to\Hom_{\DDD(\Mod\Lambda)}(-,D\Lambda)|_{\TT}$.
Then the cone $C\in\DDD(\Mod\Lambda)$ of $f$ satisfies $\Hom_{\DDD(\Mod\Lambda)}(\TT,C)=0$.
Since $\Lambda\in\TT$, we have $C=0$.
Thus $D\Lambda\simeq \S\Lambda$ belongs to $\TT$, and therefore $\projdim (D\Lambda)_\Lambda<\infty$.
On the other hand, since $\TT^{\op}$ also has a Serre functor, we have $\projdim_\Lambda(D\Lambda)<\infty$. Thus $\Lambda$ is Iwanaga-Gorenstein.

(b) This is immediate from Proposition \ref{AR duality for D} and the uniqueness of Serre functors.
\end{proof}

As an application of Proposition \ref{IwanagaGorenstein}, we give a direct proof of the observation below. Note that it was known to experts as a consequence of \cite[Theorem 3.4]{Ha} and \cite[Theorem I.2.4]{RV}.

\begin{proposition}\cite[Corollary 3.9]{C}
Let $\Lambda$ be a finite dimensional $k$-algebra. Then $\KKK^{\bo}(\proj\Lambda)$ has a Serre functor if and only if $\Lambda$ is Iwanaga-Gorenstein.
\end{proposition}

\begin{proof}
The `if' part is Proposition \ref{AR duality for D}, and the `only if' part is Proposition \ref{IwanagaGorenstein}.
\end{proof}

In the rest of this subsection, let $R$ be a $\Z$-graded Gorenstein ring in dimension $d$ such that $R=R_{\ge0}$ and $k:=R_0$ is a field, and $a$ the $a$-invariant of $R$.

The following Auslander-Reiten-Serre duality is basic.

\begin{proposition}\cite{AR,IT}\label{AR duality for CM}
Under the above setting, there is a functorial isomorphism
\begin{equation*}
\underline{\Hom}^{\Z}_R(X,Y)\simeq D\underline{\Hom}^{\Z}_R(Y,X(a)[d-1]).
\end{equation*}
for any $X\in\CM^{\Z}R$ and $Y\in\CM^{\Z}_0R$.
\end{proposition}

The results above give the following important observation.

\begin{theorem}\label{tilting functor for CM}
Under the above setting, we assume that $\underline{\CM}^{\Z}_0R$ has a tilting object $U$.
\begin{enumerate}[\rm(a)]
\item $\Lambda:=\underline{\End}_R^{\Z}(U)$ is an Iwanaga-Gorenstein ring.
\item There is a triangle equivalence $F:\underline{\CM}^{\Z}_0R\simeq\KKK^{\bo}(\proj\Lambda)$ and the following commutative diagram up to an isomorphism of functors.
\[\xymatrix@R1.5em{
\underline{\CM}^{\Z}_0R\ar[rr]^F\ar[d]^{(a)}&&\KKK^{\bo}(\proj\Lambda)\ar[d]^{\nu[1-d]}\\
\underline{\CM}^{\Z}_0R\ar[rr]^F&&\KKK^{\bo}(\proj\Lambda)
}\]
\end{enumerate}
\end{theorem}

\begin{proof}
The assertion is immediate from Propositions \ref{tilting functor}, \ref{AR duality for CM} and \ref{IwanagaGorenstein}.
\end{proof}

We give an analogue of Buchweitz's description of $\underline{\CM}^{\Z}R$ \cite{Buc} for $\underline{\CM}_0^{\Z}R$.

\begin{proposition}\label{subcategory of Dsg}
Under the above setting, let $\DD_R:=\thick\{\proj^{\Z}R,\mod_0^{\Z}R\}\subseteq\DDD^{\bo}(\mod^{\Z}R)$.
Then the triangle equivalence $\DDD^{\bo}(\mod^{\Z}R)/\KKK^{\bo}(\proj^{\Z}R)\simeq\underline{\CM}^{\Z}R$ restricts to a triangle equivalence
\[\DD_R/\KKK^{\bo}(\proj^{\Z}R)\simeq\underline{\CM}^{\Z}_0R.\]
\end{proposition}

\begin{proof}
For any $\Z$-graded prime ideal $\pp$ of $R$, the following diagram is commutative up to an isomorphism of functors.
\[\xymatrix@R1.5em{
\DDD^{\bo}(\mod^{\Z}R)/\KKK^{\bo}(\proj^{\Z}R)\ar[r]_(.7)\sim^(.7)F\ar[d]^{(-)_\pp}&\underline{\CM}^{\Z}R\ar[d]^{(-)_\pp}\\
\DDD^{\bo}(\mod^{\Z}R_\pp)/\KKK^{\bo}(\proj^{\Z}R_\pp)\ar[r]_(.7)\sim^(.7){F_\pp}&\underline{\CM}^{\Z}R_\pp}\]
Let $X\in\DD_R$. For any $\pp\neq R_{>0}$, we have $X_\pp\in\KKK^{\bo}(\proj^{\Z}R_{\pp})$ and hence $F(X)_\pp=F_\pp(X_\pp)=0$.
Thus $F(X)\in\underline{\CM}^{\Z}_0R$ holds, and hence $F$ restricts a fully faithful triangle functor $\DD_R/\KKK^{\bo}(\proj^{\Z}R)\to\underline{\CM}^{\Z}_0R$.
This is dense by \cite[Theorem 2.2]{OP}.
\end{proof}

\subsection{Basic properties of $\Z$-graded modules}
In this subsection, we assume that $R$ is a ring satisfying (R1) and (R2). 
Recall that $K$ is the $\Z$-graded total quotient ring of $R$. 
Since $R$ is Cohen-Macaulay, each associated prime ideal of $R$ is minimal. By prime avoidance, there exists a homogeneous non-zero-divisor $r\in R$ with positive degree $p>0$.

We start with the following easy observations.

\begin{lemma}\label{basic}
\begin{enumerate}[\rm(a)]
\item The inclusion functor $\Mod^{\ge0}R\to\Mod^{\Z}R$ has a right adjoint functor $(-)_{\ge0}$.
\item The restriction functor $\Mod^{\Z}K\to\Mod^{\Z}R$ has a left adjoint functor $K\otimes_R-$.
\item For any $X\in\mod^{\Z}K$, we have $K\otimes_R(X_{\ge0})=X$.
\end{enumerate}
\end{lemma}

\begin{lemma}\label{basic properties of K1}
We have $K=R[r^{-1}]$. In particular, for $i\gg0$, we have $r:R_i\simeq R_{i+p}$ and $R_i=K_i$.
\end{lemma}

\begin{proof}
To prove $K=R[r^{-1}]$, it suffices to show that each homogeneous non-zero-divisor $x\in K':=R[r^{-1}]$ is invertible.
A bijection between $\Z$-graded prime ideals $\pp$ of $R$ such that $r\notin\pp$ and $\Z$-graded prime ideals of $K'$ is given by $\pp\mapsto \pp K'$.
If $x$ is not invertible, then there exists a $\Z$-graded prime ideal $\pp$ of $R$ such that $x\in\pp K'$ and $r\notin\pp$.
Since $\pp\subsetneq R_{>0}$ and $\dim R=1$, $\pp$ is a minimal prime ideal of $R$ and hence consists of zero-divisors, a contradiction to $x\in\pp K'$.
Thus $K=R[r^{-1}]$ holds.

Since $R/Rr$ is artinian, the remaining assertions follow.
\end{proof}

To give basic properties, recall that, for $X,Y\in\mod^{\Z}R$ and $n\ge0$, $\Ext^n_R(X,Y)$ is a $\Z$-graded $R$-module whose degree $i$ part is $\Ext^n_R(X,Y)_i=\Ext^n_{\mod^{\Z}R}(X,Y(i))$.

\begin{lemma}\label{basic properties of K}
\begin{enumerate}[\rm(a)]
\item We have $R_a\subsetneq K_a$ and $R_{\ge a+1}=K_{\ge a+1}$.
\item For any $i\in\Z$, we have $K(i)\simeq K(i+p)$ and $K(i)_{\ge0}\simeq K(i+p)_{\ge0}$.
\item For any $i\in\Z$, $K(i)_{\ge0}\in\mod^{\Z}R$ holds.
\end{enumerate}
\end{lemma}

\begin{proof}
(a) Since $\Ext^1_{\mod^{\Z}R}(k(-a),R)=\Ext^1_R(k,R)_a\neq0$, there is a non-split short exact sequence $0\to R\to X\to k(-a)\to0$. Since $X\in\CM^{\Z}R$, we can regard $X\subset K$ and hence $R_a\subsetneq X_a\subseteq K_a$.

If $R_{\ge a+1}\neq K_{\ge a+1}$, then $K/R$ has $k(-i)$ as a simple submodule for some $i\ge a+1$. Thus there is a non-split short exact sequence $0\to R\to X\to k(-i)\to0$, and hence $\Ext^1_R(k,R)_i=\Ext^1_{\mod^{\Z}R}(k(-i),R)_0\neq0$, a contradiction.

(b) The multiplication map $r:K(i)\to K(i+p)$ is an isomorphism.

(c) The assertion follows from (a) and (b).
\end{proof}

Now we show the following easy observations.

\begin{proposition}\label{basic properties of K2}
\begin{enumerate}[\rm(a)]
\item $K$ is an injective object in $\mod^{\Z}K$.
\item $K(i)_{\ge0}$ is an injective object in $\mod^{\ge0}R$ for any $i\in\Z$
\end{enumerate}
\end{proposition}

\begin{proof}
(a) 
Let $X\in\mod^{\Z}K$.
Then we have $X_{\ge 0}\in\mod^{\Z}R$ by Lemma \ref{basic properties of K}(c).
Since $\dim R=1$, we have $X_{\ge0}\in\CM^{\Z}R$.
Thus $\Ext^{1}_{K}(X,K) \simeq K\otimes_R\Ext^{1}_{R}(X_{\ge 0},R)=K\otimes_R0=0$ by Lemma \ref{basic}(c).

(b) We have isomorphisms of functors on $\mod^{\ge0}R$:
\[\Hom^{\Z}_R(-,K(i)_{\ge0})\stackrel{{\rm Lem. \ref{basic}(a)}}{=}
\Hom^{\Z}_R(-,K(i))\stackrel{{\rm Lem. \ref{basic}(b)}}{=}
\Hom^{\Z}_K(K\otimes_R-,K(i)).\]
This is an exact functor since $K$ is a flat $R$-module and $K(i)$ is an injective object in $\mod^{\Z}K$ by (a).
Thus $K(i)_{\ge0}$ is injective in $\mod^{\ge0}R$.
\end{proof}

Using an isomorphism $\Ext^1_R(-,R(a))\simeq D$ of functors on $\mod^{\Z}_0R\to\mod_0^{\Z}R$, we show the following key observations.

\begin{lemma}\label{vanishing}
\begin{enumerate}[\rm(a)]
\item For all integers $i,j\in\Z$ satisfying $j<i$ and $j\le a$, we have
\[\Hom^{\Z}_R(R(i)_{\ge0},R(j)_{\ge0})=
\Hom^{\Z}_R(R(i)_{\ge0},R(j))=0.\]
\item 
Assume $a\ge0$.
For all $i>0$ and $X\in\CM^{\ge0}R$, we have $\underline{\Hom}^{\Z}_R(R(i)_{\ge0},X)=\Hom^{\Z}_R(R(i)_{\ge0},X)$.
\end{enumerate}
\end{lemma}

\begin{proof}
(a) The first equality follows from Lemma \ref{basic}(a).

We show the second equality.
Consider an exact sequence $0\to R(i)_{\ge0}\to R(i)\to M\to 0$ with $M:=R(i)/R(i)_{\ge0}\in\mod^{<0}R$.
Applying $\Hom_R(-,R(j))$, we have an exact sequence
\[\Hom^{\Z}_R(R(i),R(j))\to
\Hom^{\Z}_R(R(i)_{\ge0},R(j))\to
\Ext^1_R(M,R(j))_0.\]
Since $j<i$, we have $\Hom^{\Z}_R(R(i),R(j))=R_{j-i}=0$. Moreover
\[\Ext^1_R(M,R(j))_0=\Ext^1_R(M,R(a))_{j-a}=(DM)_{j-a}=0\]
holds by $DM\in\mod^{>0}R$ and $j-a\le0$.
Thus $\Hom^{\Z}_R(R(i)_{\ge0},R(j))=0$ holds.

(b) Clearly $\Hom^{\Z}_R(R(j),X)=0$ for any $j>0$, and $\Hom^{\Z}_R(R(i)_{\ge0},R(j))=0$ holds for any $j\le 0$ by (a).
Thus the assertion follows.
\end{proof}

For $X\in\mod R$, let ${\rm NP}(X):=\{\pp\in\Spec R\mid X_\pp\notin\proj R_\pp\}$ be the non-projective locus of $X$. Clearly ${\rm NP}(X)=\Supp\Ext^1_R(X,\Omega X)$ holds.

\begin{lemma}\label{NP}
For $X\in\mod^{\Z}R$, each minimal element in ${\rm NP}(X)$ is $\Z$-graded. In particular, $X\in\proj R$ if and only if $X_\pp\in\proj R_\pp$ for each $\Z$-graded prime ideal $\pp$ of $R$.
\end{lemma}

\begin{proof}
For $\pp\in\Spec R$, we denote by $\pp^*\in\Spec R$ the ideal generated by all homogeneous elements in $\pp$. Since $E:=\Ext^1_R(X,\Omega X)$ is a $\Z$-graded $R$-module, $\pp\in\Supp E$ if and only if $\pp^*\in\Supp E$ \cite[1.5.6]{BH}. Thus each minimal element $\pp\in{\rm NP}(X)$ satisfies $\pp=\pp^*$.
\end{proof}

We give the following description of the category $\CM_0^{\Z}R$ in \eqref{CM0}.

\begin{proposition}\label{CM0=CM02}
$\CM_0^{\Z}R=\{X\in\CM^{\Z}R\mid K\otimes_RX\in\proj K\}$.
\end{proposition}

\begin{proof}
Since $\dim R=1$, $X\in\CM^{\Z}R$ belongs to $\CM_0^{\Z}R$ if and only if $X_\pp\in\proj R_\pp$ for each minimal prime ideal $\pp$ of $R$. Applying Lemma 4.14 to $K\otimes_RX\in\mod^{\Z}K$, this is equivalent to $K\otimes_RX\in\proj K$ since $\Z$-graded prime ideals of $K$ correspond bijectively to minimal prime ideals of $R$.
\end{proof}

\subsection{Proofs of Theorem \ref{general tilting for qgr} and Corollaries \ref{grothendieck} and \ref{abundance of CM}}

Theorem \ref{general tilting for qgr} follows easily from the following standard observations.

\begin{proposition}\label{K and Lambda}
\begin{enumerate}[\rm(a)]
\item $P=\bigoplus_{i=1}^pK(i)$ is a progenerator of $\mod^{\Z}K$ such that $\End^{\Z}_R(P)\simeq\Lambda$.
\item There is an equivalence $\Hom^{\Z}_R(P,-):\mod^{\Z}K\simeq\mod\Lambda$.
\item $U=\bigoplus_{i=1}^pK(i)_{\ge0}$ is a progenerator in $\qgr R$. Therefore $U$ is a tilting object in $\per(\qgr R)$.
\item $\Lambda$ is a finite dimensional self-injective $k$-algebra.
\item If $R$ is reduced, then $\Lambda$ is a semisimple $k$-algebra. Otherwise $\Lambda$ has infinite global dimension.
\end{enumerate}
\end{proposition}

\begin{proof}
(a) Since $\{K(i)\mid i\in\Z\}$ is a progenerator of $\mod^{\Z}K$ and $K(i+p)\simeq K(i)$ holds for any $i\in\Z$, $P$ is a progenerator.
Since $\End_R(P)=\End_{K}(P)$, we have $\End^{\Z}_R(P)=\End^{\Z}_{K}(P)\simeq\Lambda$.

(b) Immediate from (a) and Morita theory.

(c) Consider the functors $(-)_{\ge0}:\mod^{\Z}K\to\mod^{\Z}R$ and $K\otimes_R-:\mod^{\Z}R\to\mod^{\Z}K$.
One can check that they induce mutually quasi-inverse equivalences $\mod^{\Z}K\simeq\qgr R$ (e.g.\ \cite[Proposition 6.21]{HIO}).
Since $P\in\mod^{\Z}K$ corresponds to $U\in\qgr R$, $U$ is a progenerator in $\qgr R$ by (a).

(d) Since $P$ is injective in $\mod^{\Z}K$ by Proposition \ref{basic properties of K2}(a), $\Lambda$ is injective in $\mod\Lambda$ by (b).

(e) $R$ is reduced if and only if $K$ is reduced if and only if any homogeneous element of $K$ is 
invertible. This is equivalent to that any object in $\mod^{\Z}K$ is projective, that is, $\gl(\mod^{\Z}K)=0$.
By (b), this is equivalent to $\Lambda$ is semisimple.

On the other hand, by a classical result of Eilenberg and Nakayama, a self-injective algebra is either semisimple or of infinite global dimension. Thus the last assertion follows from (d).
\end{proof}

We give another proof of Theorem \ref{general tilting for qgr}(a) by using Theorem \ref{embedding2}.
Note that $U$ can be written as
\[U=\bigoplus_{i=a+1}^{a+p}K(i)_{\ge 0}=
\bigoplus_{i=a+1}^{a+p}R(i)_{\ge 0}\in\DDD^{\bo}(\mod^{\Z}R).\]
Theorem \ref{general tilting for qgr}(a) is a direct consequence of the following result.

\begin{proposition}\label{W^1}
\begin{enumerate}[\rm(a)]
\item $U$ belongs to $\UU:=\DD_R^{\ge0}\cap(\DD_R^{>a})^*$.
\item $U$ is a tilting object in $\UU\simeq\per(\qgr R)$.
\end{enumerate}
\end{proposition}

\begin{proof}
(a) Since $U\in\DD_R^{\ge0}$ holds clearly, we only have to show $U^*\in\DD_R^{>a}$.
Fix $i\ge a+1$. Since $R(i)_{\ge0}\in\CM^{\Z}R$, we have $(R(i)_{\ge0})^*=\Hom_R(R(i)_{\ge0},R)$.
Since $\Hom^{\Z}_R(R(i)_{\ge0},R(j))=0$ holds for any $j\le a$ by Lemma \ref{vanishing}(a), we have $(R(i)_{\ge0})^*\in \DD_R^{>a}$ as desired.

(b) $\UU\simeq\per(\qgr R)$ holds by Theorem \ref{embedding2}(b).
We have $\Hom_{\UU}(K(i)_{\ge0},K(j)_{\ge0}[\ell])=0$ for all $\ell\neq0$ by Proposition \ref{basic properties of K2}(b).
It remains to show $\UU=\thick U$, or equivalently, $\per(\qgr R)=\thick U$.
For all $i\in\Z$, the multiplication map $r:R(i)\to R(i+p)$ is an isomorphism in $\qgr R$ since $r$ is a non-zero-divisor and hence $R/rR$ is artinian.
For all $i$ with $a<i\le a+p$, $R(i)$ belongs to $\thick U$ since $R(i)\simeq R(i)_{\ge0}$ holds in $\qgr R$.
Thus $\per(\qgr R)=\thick(\proj^{\Z}R)=\thick U$ holds.
\end{proof}

Now we prove Corollaries \ref{grothendieck} and \ref{abundance of CM}.

\begin{proof}[Proof of Corollary \ref{grothendieck}]
(a) The isomorphism classes of indecomposable projective objects in $\qgr R$ are given by $K^i(j)$ with $1\le i\le m$ and $0\le j<p_i$. Thus their number is $\sum_{i=1}^mp_i$.

(b) This follows immediately from (a) and Corollary \ref{SOD} since the Grothendieck groups of $\KKK^{\bo}(\proj^{[\ell,\ell-a-1]}A)$ for $a<0$ and $\DDD^{\bo}(\mod^{[\ell-a,\ell-1]}A)$ for $a>0$ are $\Z^{|a|}$.
\end{proof}

\begin{proof}[Proof of Corollary \ref{abundance of CM}]
Let $k=A_0$ and $k[t]$ be a polynomial ring with $\deg t=a$.
Then $R=k[t]\otimes_kA$ is a $\Z$-graded ring satisfying (R1) and (R2), and the $a$-invariant of $R$ is $0$ by \cite[Corollary 3.6.14]{BH}.
By Corollary \ref{SOD} and Theorem \ref{general tilting for qgr}, we have a triangle equivalence $\underline{\CM}_0^{\Z}R\simeq\per(\qgr R)\simeq\KKK^{\bo}(\proj\Lambda)$ for $\Lambda$ in \eqref{lambda} with $p=a$.
Since $K=k[t^{\pm1}]\otimes_kA$, it is clear that there is an equivalence $\proj^{\Z/a\Z}A\simeq\proj\Lambda$ sending $A(i)$ to the projective $\Lambda$-module given by its $i$-th row (see \cite[Theorem 3.1]{IL}).
Thus $\underline{\CM}_0^{\Z}R\simeq\KKK^{\bo}(\proj\Lambda)\simeq\KKK^{\bo}(\proj^{\Z/a\Z}A)$.
\end{proof}

\subsection{Proofs of Theorem \ref{general tilting} and Corollary \ref{general tilting reduced}}

In this subsection, we assume that the $a$-invariant $a$ of $R$ is non-negative unless otherwise stated. Let
\[\VV:=\DD_R^{\ge-a}\cap(\DD_R^{>a})^*\supseteq\UU=\DD_R^{\ge0}\cap(\DD_R^{>a})^*.\]
Then we have 
\begin{eqnarray*}
\underline{\CM}^{\Z}_0R\stackrel{}{\simeq}\DD_R/\KKK^{\bo}(\proj^{\Z}R)\stackrel{{\rm Thm.\ref{embedding2}(a)}}{\simeq}\VV
\stackrel{{\rm Thm.\ref{embedding2}(c)}}{=}\UU\perp\DDD^{\bo}(\mod^{[-a,-1]}R).
\end{eqnarray*}
We define a subalgebra of the $a\times a$ matrix algebra ${\rm M}_a(R)$ by
\[R^a:=(R_{i-j})_{1\le i,j\le a}.\]

\begin{proposition}\label{W^2}
The category $\mod^{[-a,-1]}R$ is equivalent to $\mod R^a$ and has a progenerator $\bigoplus_{i=1}^a(R/R_{\ge i})(i)\in\mod^{[-a,-1]}R$.
Thus $\DDD^{\bo}(\mod^{[-a,-1]}R)$ has a tilting object\[W:=\bigoplus_{i=1}^a(R/R_{\ge i})(i)[-1]\in\DDD^{\bo}(\mod^{[-a,-1]}R).\]
\end{proposition}

\begin{proof}
We have an equivalence $\mod^{[-a,-1]}R\simeq\mod R^a$ sending $\bigoplus_{i=-a}^{-1}X_i$ to $[X_{-1}\ X_{-2}\ \cdots\ X_{-a}]$.
Since it sends $\bigoplus_{i=1}^a(R/R_{\ge i})(i)$ to $R^a$, the first assertion follows. The second assertion is an immediate consequence.
\end{proof}

We can glue the tilting objects $U\in\UU$ and $W\in\DDD^{\bo}(\mod^{[-a,-1]}R)$ as follows.

\begin{lemma}\label{glueing tilting objects}
$\VV=\UU\perp\DDD^{\bo}(\mod^{[-a,-1]}R)$ has a tilting object $U\oplus W$.
\end{lemma}

\begin{proof}
Clearly $\UU=\thick U$ and
$\DDD^{\bo}(\mod^{[-a,-1]}R)=\thick W$ imply $\VV=\thick(U\oplus W)$.

By Propositions \ref{W^1} and \ref{W^2}, we have $\Hom_{\VV}(W,W[\ell])=0$ and $\Hom_{\VV}(U,U[\ell])=0$ for all $\ell\neq0$.
Since $\VV=\UU\perp\DDD^{\bo}(\mod^{[-a,-1]}R)$, we have $\Hom_{\VV}(U,W[\ell])=0$ for all $\ell\in\Z$.

It remains to check $\Hom_{\VV}((R/R_{\ge j})(j)[-1],K(i)_{\ge0}[\ell])=0$ for all $\ell\neq0$.
If $\ell<-1$, then this is clear since $(R/R_{\ge j})(j)$ and $K(i)_{\ge0}$ are modules.
If $\ell=-1$, then this follows from $(R/R_{\ge j})(j)\in\mod^{\Z}_0R$ and $K(i)_{\ge0}\in\CM^{\Z}R$.
Assume $\ell>0$. Since the syzygy of $(R/R_{\ge j})(j)$ is $R(j)_{\ge0}$, we have
\[\Hom_{\VV}((R/R_{\ge j})(j)[-1],K(i)_{\ge0}[\ell])=\Ext^\ell_R(R(j)_{\ge0},K(i)_{\ge0})_0\stackrel{{\rm Prop.\ref{basic properties of K2}(b)}}{=}0.\qedhere\]
\end{proof}

We are ready to prove Theorem \ref{general tilting}.

\begin{proof}[Proof of Theorem \ref{general tilting}]
(a) This follows from Lemma \ref{glueing tilting objects} since $V\simeq U\oplus W$ in $\underline{\CM}_0^{\Z}R$.

(b) Immediate from (a) and Proposition \ref{tilting functor}.

(c) The triangle equivalence $\VV\simeq\underline{\CM}_0^{\Z}R$ sends $\bigoplus_{i=1}^a(R/R_{\ge i})(i)[-1]$ to $\bigoplus_{i=1}^aR(i)_{\ge0}$. Thus
\[\underline{\End}^{\Z}_R(\bigoplus_{i=1}^aR(i)_{\ge0})
\simeq\End^{\Z}_R(\bigoplus_{i=1}^a(R/R_{\ge i})(i))=R^a.\]
Hence the left upper entries of \eqref{gamma} are correct.
The right upper entries are also correct since $\Hom_R^{\Z}(R(i)_{\ge0},R(j)_{\ge0})=0$ holds for all $j\le a<i$ by Lemma \ref{vanishing}(a).
Finally, the lower entries are correct since for all $a+1\le j\le a+p$, we have
\begin{eqnarray*}
\underline{\Hom}^{\Z}_R(R(i)_{\ge0},R(j)_{\ge0})
&\stackrel{{\rm Lem.\ref{vanishing}(b)}}{=}&\Hom^{\Z}_R(R(i)_{\ge0},R(j)_{\ge0})
\stackrel{{\rm Lem.\ref{basic properties of K}(a)}}{=}\Hom^{\Z}_R(R(i)_{\ge0},K(j)_{\ge0})\\
&\stackrel{{\rm Lem.\ref{basic}(a)}}{=}&\Hom^{\Z}_R(R(i)_{\ge0},K(j))
\stackrel{{\rm Lem.\ref{basic}(b)}}{=}\Hom^{\Z}_K(K(i),K(j))=K_{j-i}.
\end{eqnarray*}

(d) This follows from Theorem \ref{tilting functor for CM}(a).

(e) For the triangular matrix ring $A=\left[\begin{smallmatrix}B&0\\ M&C\end{smallmatrix}\right]$, it is well known that
\[\max\{\gl B,\gl C\}\le\gl A\le\gl B+\gl C+1\]
holds. Applying it repeatedly, we obtain $\gl R^a<\infty$. Since $\Gamma$ has a form
$\left[\begin{smallmatrix}
R^a&0\\
M&\Lambda
\end{smallmatrix}\right]$, it follows that $\gl\Gamma<\infty$ if and only if $\gl\Lambda<\infty$.
Thus the assertion follows from Theorem \ref{general tilting for qgr}(e).
\end{proof}

To prove Corollary \ref{general tilting reduced}, we prepare the following.

\begin{proposition}\label{negative and reduced imply regular}
Under the setting (R1) and (R2), if $a<0$ and $R$ is reduced, then $R\simeq k[t]$.
\end{proposition}

\begin{proof}
Since $R$ is reduced, $K$ is a product $k^1[t_1^{\pm1}]\times\cdots\times k^m[t_m^{\pm1}]$ of Laurent polynomial algebras over field extensions $k^i$ of $k$ \cite[Lemma 1.5.7]{BH}.
Since $a<0$, we have $R=K_{\ge0}$ by  Lemma \ref{basic properties of K}(a).
Thus $K_0=R_0=k$ holds, and therefore $K=k[t^{\pm1}]$ and $R=K_{\ge0}=k[t]$.
\end{proof}

We are ready to prove Corollary \ref{general tilting reduced}.

\begin{proof}[Proof of Corollary \ref{general tilting reduced}]
(a) is shown in Proposition \ref{negative and reduced imply regular},
and (b) and (c) are immediate from Theorem \ref{general tilting}.
Now (d) is clear from the shape of $\Gamma$ in \eqref{gamma}.
\end{proof}

\subsection{Proof of Theorem \ref{negative case}}

We start with the following general result for `silting reduction' of triangulated categories.

\begin{proposition}\cite[Theorem 4.8]{IY1}\label{silting reduction}
Let $\UU$ be a triangulated category with a silting object $U$. For any $P\in\add U$, the Verdier quotient $\UU/\thick P$ has a silting object $U$.
\end{proposition}

In the rest, we assume $a<0$. Now we prove Theorem \ref{negative case}.

\begin{proof}[Proof of Theorem \ref{negative case}]
(a)(b) Let $\UU=\DD_R^{\ge0}\cap(\DD_R^{>a})^*$.
By Theorem \ref{embedding2}(a)(c), there are triangle equivalences
\[{\displaystyle\underline{\CM}_0^{\Z}R\simeq\DD_R^{\ge-a}\cap(\DD_R^{>a})^*\simeq\frac{\UU}{\KKK^{\bo}(\proj^{[0,-a-1]}R)}=\frac{\UU}{\thick P}}\]
for $P=\bigoplus_{i=a+1}^{0}R(i)$.
By Proposition \ref{K and Lambda}(c), the triangulated category $\UU$ has a tilting object $U=\bigoplus_{i=a+1}^{a+p}R(i)_{\ge0}$.
Applying Proposition \ref{silting reduction} to $\UU$ and the direct summand $P$ of $U$, it follows that $U$ is a silting object in $\UU/\thick P\simeq\underline{\CM}_0^{\Z}R$.

(c) 
Assume that $R$ is not regular and that $\underline{\CM}^{\Z}_0R$ has a tilting object $T$. Let $\Lambda=\underline{\End}_R^{\Z}(T)$.
By Theorem \ref{tilting functor for CM}, there is a triangle equivalence $F:\underline{\CM}_0^{\Z}R\simeq\KKK^{\bo}(\proj \Lambda)$ sending $T$ to $\Lambda$ and making the following diagram commutative.
\[\xymatrix@R1.5em{
\underline{\CM}_0^{\Z}R\ar[rr]^F\ar[d]^{(a)}&&\KKK^{\bo}(\proj\Lambda)\ar[d]^{\nu}\\
\underline{\CM}_0^{\Z}R\ar[rr]^F&&\KKK^{\bo}(\proj\Lambda)
}\]
For all $\ell\ge0$, $\nu^{\ell}(\Lambda)\in\DDD^{\le0}(\mod\Lambda)$ holds clearly, and hence $H^i(\nu^{\ell}(\Lambda))=0$ holds for all $i>0$.

On the other hand, take a surjective morphism $f:\bigoplus_{j=1}^nR(-b_j)\to T$ in $\mod^{\Z}R$, and let
\[s:=\min\{b_j\mid 1\le j\le n\}\le t:=\max\{b_j\mid1\le j\le n\}.\]
Then $(\Omega^iT)_{<s}=0$ holds for all $i\ge0$.
Since $a<0$, there exists $\ell\gg0$ such that $t<s-\ell a$. Then for all $i\ge0$, we have
$(\Omega^iT(\ell a))_{\le t}=0$ and hence
\[\Hom^{\Z}_R(T,\Omega^iT(\ell a))\subset\Hom^{\Z}_R(\bigoplus_{j=1}^nR(-b_j),\Omega^iT(\ell a))=\bigoplus_{j=1}^n(\Omega^iT(\ell a))_{b_j}=0.\]
Thus
$H^{-i}(\nu^{\ell}(\Lambda))=\Hom_{\DDD^{\bo}(\mod \Lambda)}(\Lambda,\nu^{\ell}(\Lambda)[-i])=\underline{\Hom}^{\Z}_R(T,\Omega^iT(\ell a))=0$ holds for all $i\ge0$.

Therefore for $\ell\gg0$, $\nu^{\ell}(\Lambda)$ is acyclic and hence zero in $\DDD^{\bo}(\mod\Lambda)$. This is a contradiction since $\nu$ is an autoequivalence. 
\end{proof}

\subsection{Proof of Theorem \ref{example hypersurface}}\label{subsection proof of HS}

(a) 
Since $\Gamma$ is Iwanaga-Gorenstein by Theorem \ref{general tilting}(d),
it suffices to show $\projdim_\Gamma(D\Gamma)\leq 2$.
Recall that $\Gamma$ has the following form.
{\small\begin{equation*}
\Gamma=
\begin{bmatrix}
R_0&0&\cdots&0&0&
0\\
R_1&R_0&\cdots&0&0&
0\\
\vdots&\vdots&\ddots&\vdots&\vdots
&\vdots\\
R_{a-2}&R_{a-3}&\cdots&R_0&0&
0\\
R_{a-1}&R_{a-2}&\cdots&R_1&R_0&
0\\
K_a&K_{a-1}&\cdots&K_2&K_1&
K_0
\end{bmatrix}.
\end{equation*}}
For $1\leq i \leq a+1$, let $e_i\in\Gamma$ be the element whose $(i,i)$-entry is $1$ and the others are $0$, and let $P^i=\Gamma e_i$ be the projective $\Gamma^{\op}$-module corresponding to the $i$-th column.

First we claim that the simple $\Gamma^{\op}$-module $S^i=P^i/{\rm rad}P^i$ has projective dimension at most $2$ for $1 \leq i \leq a$.
More precisely, we show that the sequence
\begin{eqnarray}
0\to P^{i+2} \xrightarrow{{\tiny {}^t[y \ -x]}} (P^{i+1})^{\oplus 2}\xrightarrow{{\tiny[x \ y]}} P^{i}\to S^i \to 0 \label{projresol_simple}
\end{eqnarray}
is exact (where $P^{a+2}=0$). Indeed there is an exact sequence 
\begin{eqnarray}
0 \to k[x,y](-2)\xrightarrow{{\tiny {}^t[y \ -x]}}k[x,y](-1)^{\oplus 2}\xrightarrow{{\tiny[x \ y]}}k[x,y]\to k \to 0 \label{Koszulcomplex}
\end{eqnarray}
in $\mod^{\Z}k[x,y]$.
By $k[x,y]_{<n}=R_{<n}$, the degree $i$-part
\[
0 \to R_{i+2}\to (R_{i+1})^{\oplus 2}\to R_i \to 0
\]
of \eqref{Koszulcomplex} is exact for $1\leq i <a=n-2$.
Moreover applying $-\otimes_{k[x,y]}K$ to \eqref{Koszulcomplex}, we have exact sequences $0\to K(-2)\to K(-1)^{\oplus 2}\to K\to 0$ and
\[
0 \to K_{i+2}\to (K_{i+1})^{\oplus 2}\to K_i \to 0
\]
for $i \in \Z$.
Thus \eqref{projresol_simple} is exact since each entry is exact by the above two exact sequences.

By the above claim, any $\Gamma^{\op}$-module annihilated by $e_{a+1}$ has projective dimension at most $2$.
In particular we have $\projdim_\Gamma D(e_i\Gamma)\leq 2$ for $1\leq i\leq a$. 
Finally there is an exact sequence
\[
0\to {}^t\, [0 \ 0\ \cdots \ 0 \ DK_0]\to D(e_{a+1}\Gamma)\to{}^t\, [DK_a \ DK_{a-1} \  \cdots \ DK_1 \ 0]\to 0.
\]
The left term is isomorphic to $\Gamma e_{a+1}$ since $K_0$ is a self-injective $k$-algebra.
The right term has projective dimension at most $2$ since it is annihilated by $e_{a+1}$.
Thus $\projdim_\Gamma D(e_{a+1}\Gamma)\leq 2$ holds, and we have the desired inequality.

(b) 
Since $R$ is a hypersurface, $[2]=(n)$ holds. Since $R$ has $a$-invariant $n-2$, our triangulated category $\underline{\CM}_0^{\Z}R$ has a Serre functor $\S=(n-2)$ by Proposition \ref{AR duality for CM}.
Thus $\S^n\simeq[2(n-2)]$ holds, and $\underline{\CM}_0^{\Z}R$ is a $\frac{2(n-2)}{n}$-Calabi-Yau triangulated category.
If $\Gamma'$ is derived equivalent to $\Gamma$, then $\KKK^{\bo}(\proj\Gamma')\simeq\underline{\CM}_0^{\Z}R$ holds, and hence
$1\le\frac{2(n-2)}{n}<\injdim\Gamma'_{\Gamma'}$ by \cite[Proposition 1.10(c)]{HIMO}.

(c)
Since $R^i=k[x,y]/(f_i^{n_i})$, we have a monomorphism $R\subseteq R^1\times\cdots\times R^m$ of $\Z$-graded rings whose cokernel has finite length as an $R$-module.
Thus we have an isomorphism $K\simeq K^1\times\cdots\times K^m$ of $\Z$-graded rings.
It restricts to an isomorphism $K_{\ge0}\simeq K^1_{\ge0}\times\cdots\times K^m_{\ge0}$ of $\Z$-graded rings and of $\Z$-graded $R$-modules. We show that $K^i_{\ge0}$ is indecomposable in $\CM^{\Z}R$.

By our choice, $g_i=\alpha'_ix+\beta'_iy\in k[x,y]$ is a non-zero-divisor of $R^i$, and by Lemma \ref{basic properties of K1}, we have $K^i=R^i[g_i^{-1}]=k[h_i,g_i^{\pm1}]$ for $h_i=f_i/g_i$.
We have isomorphisms
\[K^i=K^i_0[g_i^{\pm1}]\ \mbox{ and }\ K^i_0\simeq k[b_i]/(b_i^{n_i})\]
for the polynomial ring $k[b_i]$, where $b_i$ corresponds to $h_i$.
Since $\End_R^{\Z}(K^i_{\ge 0})=(K^i_{\ge 0})_0=K^i_0=k[b_i]/(b_i^{n_i})$ is a local algebra, $K^i_{\ge0}$ is indecomposable.

(d)
The Jacobson radical of $\Gamma$ and its square are
\begin{equation*}
{\small {\rm rad}\Gamma=}
\left[\begin{smallmatrix}
0&0&0&\cdots&0&0&0&0\\
R_1&0&0&\cdots&0&0&0&0\\
R_2&R_1&0&\cdots&0&0&0&0\\
R_3&R_2&R_1&\cdots&0&0&0&0\\
\vdots&\vdots&\vdots&\ddots&\vdots&\vdots&\vdots&\vdots\\
R_{a-2}&R_{a-3}&R_{a-4}&\cdots&R_1&0&0&0\\
R_{a-1}&R_{a-2}&R_{a-3}&\cdots&R_2&R_1&0&0\\
K_{a}&K_{a-1}&K_{a-2}&\cdots&K_3&K_2&K_1&{\rm rad}K_0
\end{smallmatrix}\right],\ {\small {\rm rad}^2\Gamma=}
\left[\begin{smallmatrix}
0&0&0&\cdots&0&0&0&0\\
0&0&0&\cdots&0&0&0&0\\
R_2&0&0&\cdots&0&0&0&0\\
R_3&R_2&0&\cdots&0&0&0&0\\
\vdots&\vdots&\vdots&\ddots&\vdots&\vdots&\vdots&\vdots\\
R_{a-2}&R_{a-3}&R_{a-4}&\cdots&0&0&0&0\\
R_{a-1}&R_{a-2}&R_{a-3}&\cdots&R_2&0&0&0\\
K_{a}&K_{a-1}&K_{a-2}&\cdots&K_3&K_2&{\rm rad}K_1&{\rm rad}^2K_0\end{smallmatrix}\right].
\end{equation*}
Thus ${\rm rad}\Gamma/{\rm rad}^2\Gamma$ is a direct sum of the following:
\begin{itemize}
\item
$e_{i+1}(\frac{{\rm rad}\Gamma}{{\rm rad}^2\Gamma})e_i=R_1=\langle x,y \rangle_k \hspace{5mm} (1\le i \le a-1)$. \vspace{2mm}
\item
$e_{a+1}(\frac{{\rm rad}\Gamma}{{\rm rad}^2\Gamma})e_a
=\frac{K_1}{{\rm rad} K_1}=\frac{K_1^1}{{\rm rad} K_1^1}\times\cdots\times\frac{K_1^m}{{\rm rad} K_1^m}
=\frac{k[b_1]}{(b_1)}g_1\times\cdots\times\frac{k[b_m]}{(b_m)}g_m$ \vspace{2mm}
\item 
$e_{a+1}(\frac{{\rm rad}\Gamma}{{\rm rad}^2\Gamma})e_{a+1}
=\frac{{\rm rad} K_0}{{\rm rad}^2 K_0}=\frac{{\rm rad} K_0^1}{{\rm rad}^2 K_0^1}\times\cdots\times\frac{{\rm rad} K_0^m}{{\rm rad}^2 K_0^m}
=\frac{(b_1)}{(b_1^2)}\times\cdots\times\frac{(b_m)}{(b_m^2)}$.
\end{itemize}
Thus we obtain the quiver of $\Gamma$ as in the assertion. The proof of relations are direct and left to the reader.

(e) Clear from (d) and (a).
\qed

\subsection{Proof of Theorem \ref{simple singularity}}\label{subsection proof of simple singularity}

We prove Theorem \ref{simple singularity} by applying Theorem \ref{general tilting} and a general recipe to calculate mutation \cite{AI} given by Mizuno \cite[Theorem 1.2]{M}.
We denote by $V$ and $\Gamma$ the tilting object and its endomorphism algebra given in
Theorem \ref{general tilting}.

($A_{2n-1}$) Let $R=k[x,y]/(x^{2n}-y^2)$ with $\deg x=1$ and $\deg y=n$, so $a=n-1$.
Then $K=k[t^{\pm1}]\times k[u^{\pm1}]$ with $\deg t=\deg u=1$, $x=t+u$ and $y=t^n-u^n$, so $p=1$.
Our $V$ is $(\bigoplus_{i=1}^{n-1}R(i)_{\ge0})\oplus k[t]\oplus k[u]$,
and $\Gamma$ is the path algebra of type $D_{n+1}$:
{\small\[\xymatrix@R=0em{
&&&&&k[t]\\
R(1)_{\ge0}\ar[r]^x&R(2)_{\ge0}\ar[r]^x&\cdots\ar[r]^(.4)x&R(n-2)_{\ge0}\ar[r]^x&R(n-1)_{\ge0}\ar[ru]^(.6)t\ar[rd]^(.6)u\\
&&&&&k[u].
}\]}

($A_{2n}$) Let $R=k[x,y]/(x^{2n+1}-y^2)$ with $\deg x=2$ and $\deg y=2n+1$, so $a=2n-1$.
Then $K=k[t^{\pm1}]$ with $\deg t=1$, $x=t^2$ and $y=t^{2n+1}$, so $p=1$.
Our $V$ is $\bigoplus_{i=1}^{2n}R(i)_{\ge0}$, and $\Gamma$ is the path algebra of type $A_{2n}$:
{\small\[\xymatrix@R=1em{
R(1)_{\ge0}\ar[r]^x&R(3)_{\ge0}\ar[r]^x&\cdots\ar[r]^(.4)x&R(2n-3)_{\ge0}\ar[r]^x&R(2n-1)_{\ge0}\ar[d]^t\\
R(2)_{\ge0}\ar[r]^x&R(4)_{\ge0}\ar[r]^x&\cdots\ar[r]^(.4)x&R(2n-2)_{\ge0}\ar[r]^x&R(2n)_{\ge0}.
}\]}

($D_{2n+1}$) Let $R=k[x,y]/(x^{2n}-xy^2)$ with $\deg x=2$ and $\deg y=2n-1$, so $a=2n-1$. Then 
$K=k[t^{\pm1}]\times k[u^{\pm1}]$ with $\deg t=1$, $\deg u=2n-1$, $x=t^2$, $y=t^{2n-1}+u$, so $p=2n-1$.
Our $V$ is $(\bigoplus_{i=1}^{2n-1}R(i)_{\ge0})\oplus k[t]\oplus(\bigoplus_{i=1}^{2n-1}k[u](i)_{\ge0})$,
and $\Gamma$ is
{\small\[\xymatrix@R=1em{
k[u](1)_{\ge0}&k[u](3)_{\ge0}\ar@{.}[dl]&\ar@{.}[dl]&k[u](2n-3)_{\ge0}\ar@{.}[dl]&k[u](2n-1)_{\ge0}\ar@{.}[dl]\\
R(1)_{\ge0}\ar[r]^x\ar[u]&R(3)_{\ge0}\ar[r]^x\ar[u]&\cdots\ar[r]^(.4)x&R(2n-3)_{\ge0}\ar[r]^x\ar[u]&R(2n-1)_{\ge0}\ar[d]^t\ar[u]\\
R(2)_{\ge0}\ar[r]^x\ar[d]&R(4)_{\ge0}\ar[r]^x\ar[d]&\cdots\ar[r]^(.4)x&R(2n-2)_{\ge0}\ar[r]^x\ar[d]&k[t]\\
k[u](2)_{\ge0}&k[u](4)_{\ge0}\ar@{.}[ul]&\ar@{.}[ul]&k[u](2n-2)_{\ge0}.\ar@{.}[ul]
}\]}
This is derived equivalent to the path algebra of type $A_{4n-1}$
by mutating the summands $k[u](i)_{\ge0}$ with $1\le i\le 2n-1$:
{\small\[\xymatrix@R=1em{
\bullet&\bullet\ar@{<-}[dl]&\ar@{<-}[dl]&\bullet\ar@{<-}[dl]&\bullet\ar@{<-}[dl]\\
R(1)_{\ge0}\ar@{<-}[u]&R(3)_{\ge0}\ar@{<-}[u]&\cdots&R(2n-3)_{\ge0}\ar@{<-}[u]&R(2n-1)_{\ge0}\ar[d]\ar@{<-}[u]\\
R(2)_{\ge0}\ar@{<-}[d]&R(4)_{\ge0}\ar@{<-}[d]&\cdots&R(2n-2)_{\ge0}\ar[r]\ar@{<-}[d]&k[t]\\
\bullet&\bullet\ar@{<-}[ul]&\ar@{<-}[ul]&\bullet\ar@{<-}[ul].}\]}

($D_{2n}$) Let $R=k[x,y]/(x^{2n-1}-xy^2)$ with $\deg x=1$ and $\deg y=n-1$, so $a=n-1$.
Then $K=k[t^{\pm1}]\times k[u^{\pm1}]\times k[v^{\pm1}]$ with $x=t+u$, $y=t^{n-1}-u^{n-1}+v$, $\deg t=\deg u=1$ and $\deg v=n-1$, so $p=n-1$.
Our $V$ is $(\bigoplus_{i=1}^{n-1}R(i)_{\ge0})\oplus k[t]\oplus k[u]\oplus(\bigoplus_{i=1}^{n-1}k[v](i)_{\ge0})$, and $\Gamma$ is
{\small\[\xymatrix@R=1em{
R(1)_{\ge0}\ar[r]^x\ar[d]&R(2)_{\ge0}\ar[r]^x\ar[d]&\cdots\ar[r]^(.4)x&R(n-2)_{\ge0}\ar[r]^x\ar[d]&R(n-1)_{\ge0}\ar[r]\ar[d]\ar[dr]&k[t]\\
k[v](1)_{\ge0}&k[v](2)_{\ge0}\ar@{.}[ul]&\ar@{.}[ul]&k[v](n-2)_{\ge0}\ar@{.}[ul]&k[v](n-1)_{\ge0}\ar@{.}[ul]&k[u].
}\]}
This is derived equivalent to the path algebra of type $D_{2n}$
by mutating the summands $k[v](i)_{\ge0}$ with $1\le i\le n-1$:
{\small\[\xymatrix@R=1em{
R(1)_{\ge0}\ar@{<-}[d]&R(2)_{\ge0}\ar@{<-}[d]&\cdots&R(n-2)_{\ge0}\ar@{<-}[d]&R(n-1)_{\ge0}\ar@{<-}[d]\ar[dr]\ar[r]&k[t]\\
\bullet&\bullet\ar@{<-}[ul]&\ar@{<-}[ul]&\bullet\ar@{<-}[ul]&\bullet\ar@{<-}[ul]&k[u].
}\]}

($E_6$) Let $R=k[x,y]/(x^4-y^3)$ with $\deg x=3$ and $\deg y=4$, so $a=5$.
Then $K=k[t^{\pm1}]$ with $\deg t=1$, $x=t^3$ and $y=t^4$, so $p=1$.
Our $V$ is $\bigoplus_{i=1}^6R(i)_{\ge0}$, and $\Gamma$ is
{\small\[\xymatrix@R=1.5em{
&R(1)_{\ge0}\ar[r]^x\ar[d]^y&R(4)_{\ge0}\ar[d]^{t^2}\\
R(2)_{\ge0}\ar[r]^x&R(5)_{\ge0}\ar[r]^t&R(6)_{\ge0}\ar@{.}[ul]&R(3)_{\ge0}\ar[l]_x.
}\]}
This is derived equivalent to the path algebra of type $E_6$ by mutating the summand $R(1)_{\ge0}$:
{\small\[\xymatrix@R=1.5em{
&\bullet\ar@{<-}[r]\ar@{<-}[d]&R(4)_{\ge0}\\
R(2)_{\ge0}\ar[r]&R(5)_{\ge0}&R(6)_{\ge0}\ar@{<-}[ul]&R(3)_{\ge0}.\ar[l]
}\]}

($E_7$) Let $R=k[x,y]/(x^3y-y^3)$ with $\deg x=2$ and $\deg y=3$, so $a=4$.
Then $K=k[t^{\pm1}]\times k[u^{\pm1}]$ with $\deg t=1$, $\deg u=2$, $x=t^2+u$ and $y=t^3$, so $p=2$.
Our $V$ is $(\bigoplus_{i=1}^4R(i)_{\ge0})\oplus k[t]\oplus k[u]\oplus k[u](1)_{\ge0}$, and $\Gamma$ is
{\small\[\xymatrix@R=1.5em{
k[u]\ar@{.}[r]&R(1)_{\ge0}\ar[r]^x\ar[d]^y&R(3)_{\ge0}\ar[r]^u\ar[d]^{t^2}&k[u](1)_{\ge0}\\
R(2)_{\ge0}\ar[r]^x&R(4)_{\ge0}\ar[r]^t\ar[lu]_u&k[t].\ar@{.}[lu]
}\]}
This is derived equivalent to the path algebra of type $E_7$ by successively mutating the summands $R(1)_{\ge0}$ and $k[u]$:
{\small\[\xymatrix@R=1.5em{
\bullet\ar[r]\ar@/^15pt/@{<-}[rr]&\bullet\ar@{<-}[d]&R(3)_{\ge0}\ar[r]&k[u](1)_{\ge0}\\
R(2)_{\ge0}\ar[r]&R(4)_{\ge0}&k[t].\ar@{<-}[lu]
}\]}

($E_8$) Let $R=k[x,y]/(x^5-y^3)$ with $\deg x=3$ and $\deg y=5$, so $a=7$.
Then $K=k[t^{\pm1}]$ with $\deg t=1$, $x=t^3$ and $y=t^5$, so $p=1$.
Our $V$ is $\bigoplus_{i=1}^8R(i)_{\ge0}$, and $\Gamma$ is
{\small\[\xymatrix@R=1.5em{
R(1)_{\ge0}\ar[r]^x\ar[dr]_y&R(4)_{\ge0}\ar[r]^x&R(7)_{\ge0}\ar[d]^t&R(2)_{\ge0}\ar[l]_y\ar[d]^x\\
R(3)_{\ge0}\ar[r]_x&R(6)_{\ge0}\ar[r]_{t^2}&R(8)_{\ge0}\ar@{.}[ur]\ar@{.}[ull]&R(5)_{\ge0}\ar[l]^x.
}\]}
This is derived equivalent to the path algebra of type $E_8$ by successively mutating the summands $R(1)_{\ge0}$, $R(4)_{\ge0}$, $R(8)_{\ge0}$ and $R(2)_{\ge0}$:
{\small\[\xymatrix@R=1.5em{
\bullet\ar[r]\ar@{<-}[dr]&\bullet\ar@/^15pt/@{<-}[rr]&R(7)_{\ge0}&\bullet\ar@{<-}[l]\\
R(3)_{\ge0}\ar[r]&R(6)_{\ge0}&\bullet\ar[ur]&R(5)_{\ge0}.\ar@{<-}[l].}\qedhere\]}
\vskip-1.5em\qed

\subsection{Proof of Proposition \ref{Tpq singularity}}\label{proof for Tpq}
We prove Proposition \ref{Tpq singularity} by applying Theorem \ref{general tilting} and mutation \cite{AI,M}. We omit the details of calculations. 

(a) We apply Theorem \ref{example hypersurface}. 
Let $K^i$ be the $\mathbb{Z}$-graded total quotient ring of $k[x,y]/(x-\alpha_i y)$.
Since $a=2$ in this case, our $V$ is $(\bigoplus_{i=1}^{2}R(i)_{\geq0})\oplus(\bigoplus_{i=1}^{4}K^i_{\geq0})$, and $\Gamma$ is presented by the quiver with relations
\[{\small\xymatrix@C=4em@R=0em{
& & K^1_{\geq0} \\
& & K^2_{\geq0} \\
R(1)_{\geq0} \ar@/^0.4pc/[r]^x \ar@/_0.4pc/[r]_y & R(2)_{\geq0} \ar[ruu]^{a_1} \ar[ru]|{a_2} \ar[rd]|{a_3} \ar[rdd]_{a_4} && a_i(x-\alpha_i y)=0\ (1\le i\le 4).\\
& & K^3_{\geq0} \\
& & K^4_{\geq0}
}}\]
By mutating the summand $R(2)_{\geq 0}$, we obtain a tilting object in $\underline{\CM}^{\Z}R$ whose endomorphism algebra is the following canonical algebra, where $\lambda=(\alpha_1-\alpha_4)(\alpha_2-\alpha_4)(\alpha_1-\alpha_3)^{-1}(\alpha_2-\alpha_4)^{-1}$ (see also \cite[Figure 1.1]{Jas}).
\[{\small\xymatrix@C=4em@R=0em{
 & K^1_{\geq0} \ar[rdd]^{b_1}& \\
 & K^2_{\geq0} \ar[rd]|{b_2} &&b_1a_1+b_2a_2+b_3a_3=0 \\
 R(1)_{\geq0} \ar[ruu]^{a_1} \ar[ru]|{a_2} \ar[rd]|{a_3} \ar[rdd]_{a_4} & & \bullet &b_1a_1+\lambda b_2a_2+b_4a_4=0.\\
 & K^3_{\geq0} \ar[ru]|{b_3}&\\
  & K^4_{\geq0} \ar[ruu]_{b_4}&
}}\]

(b) Let $K^i$ be the $\mathbb{Z}$-graded total quotient ring of $k[x,y]/(x-\alpha_i y^2)$.
Since $a=3$ in this case, our $V$ is $(\bigoplus_{i=1}^{3}R(i)_{\geq0})\oplus(\bigoplus_{i=1}^{3}K^i_{\geq0})$, and $\Gamma$ is presented by the quiver with relations
\[{\small\xymatrix@C=3em@R=0em{
& & & K^1_{\geq0} \\
R(1)_{\geq0} \ar[r]_y \ar@/^15pt/[rr]^x & R(2)_{\geq0} \ar[r]_y & R(3)_{\geq0} \ar[ru]^{a_1} \ar[r]|{a_2} \ar[rd]_{a_3} & K^2_{\geq0} &a_i(x-\alpha_iy^2)=0\ (1\le i\le 3).\\
& & & K^3_{\geq0}
}}\]
By mutating the summand $R(2)_{\geq 0}$, we obtain a tilting object in $\underline{\CM}^{\Z}R$ whose endomorphism algebra is presented by the quiver with relations
\[{\small\xymatrix@C=4em@R=0em{
& & K^1_{\geq0} \\
& & K^2_{\geq0} &a_i(x-\alpha_i y^2)=0\ (1\le i\le 3)\\
R(1)_{\geq0} \ar@/^0.4pc/[r]^x \ar@/_0.4pc/[r]_{y^2} & R(3)_{\geq0} \ar[ruu]^{a_1} \ar[ru]|{a_2} \ar[rd]|{a_3} \ar[rdd]_{a_4} && a_4y^2=0.\\
& & K^3_{\geq0} \\
& & \bullet
}}\]
As in the case (a), by mutating the summand $R(3)_{\geq 0}$, we obtain a tilting object in $\underline{\CM}^{\Z}R$ whose endomorphism algebra is the following canonical algebra, where $\lambda=(\alpha_2-\alpha_3)(\alpha_1-\alpha_3)^{-1}$.
\[{\small\xymatrix@C=4em@R=0em{
 & K^1_{\geq0} \ar[rdd]^{b_1}& \\
 & K^2_{\geq0} \ar[rd]|{b_2} &&b_1a_1+b_2a_2+b_3a_3=0\\
 R(1)_{\geq0} \ar[ruu]^{a_1} \ar[ru]|{a_2} \ar[rd]|{a_3} \ar[rdd]_{a_4} & & \bullet &b_1a_1+\lambda b_2a_2+b_4a_4=0.\\
 & K^3_{\geq0} \ar[ru]|{b_3}&\\
  & \bullet \ar[ruu]_{b_4}&&
}}\]
\vskip-1.9em\qed

\subsection{Proof of Proposition \ref{no tilting}}\label{subsection proof of no tilting}

We start with the following calculation of DG algebras.

\begin{proposition}\label{k[w]/(w^2)}
Let $\Lambda$ be the algebra in Proposition \ref{no tilting}(b),
and $P$ the projective $\Lambda$-module in Proposition \ref{no tilting}(d). Then there is a triangle equivalence $\KKK^{\bo}(\proj\Lambda)/\thick P\simeq\per k[w]/(w^2)$ for the DG algebra in Proposition \ref{no tilting}(e).
\end{proposition}

\begin{proof}
Let $M$ be the complex
\[
\cdots\to0\to P^{1}\xrightarrow{z} P^{2}\xrightarrow{z}\cdots\xrightarrow{z}P^{n-2}\xrightarrow{z} P^{n-1}\xrightarrow{z} P^n\to0\to\cdots
\]
in $\KKK^{\bo}(\proj\Lambda)$ whose non-zero degrees are $1-n,2-n,\ldots,0$.
Then $\Hom_{\KKK^{\bo}(\proj\Lambda)}(P[i],M)=0$ holds for any $i\in\Z$, and there is a triangle $N\to P^n \to M\to N[1]$ with $N\in\thick P$.
Thus we have
\[\KKK^{\bo}(\proj\Lambda)=\thick(P\oplus P^n)=\thick(P\oplus M)=(\thick P)\perp(\thick M),\]
and therefore $\KKK^{\bo}(\proj\Lambda)/\thick P\simeq\thick M$.
By \cite[Theorem 4.3]{Ke}, we have a triangle equivalence 
$\thick M\simeq\per\mathcal{E}{\rm nd}_\Lambda(M)$ for the endomorphism DG algebra $\mathcal{E}{\rm nd}_\Lambda(M)$ of $M$.
One can easily verify
\[\mathcal{E}{\rm nd}_\Lambda(M)^i=
\begin{cases}
\Hom_{\Lambda}(P^{n},P^1) & (i=1-n) \\
\bigoplus_{i=1}^n\End_{\Lambda}(P^i)& (i=0) \\
\bigoplus_{i=1}^{n-1}\Hom_{\Lambda}(P^i,P^{i+1})& (i=1) \\
0 & (\mbox{otherwise}),
\end{cases}
\vspace{2mm}
\]
and there is a quasi-isomorphism $k[w]/(w^2)\to\mathcal{E}{\rm nd}_\Lambda(M)$ of DG algebras given by $w\mapsto(z:P^{n}\to P^1)$. Thus
\[
\KKK^{\bo}(\proj\Lambda)/\thick P\simeq\thick M\simeq 
\per\mathcal{E}{\rm nd}_\Lambda(M)\simeq\per k[w]/(w^2).\qedhere
\]
\end{proof}

We are ready to prove Proposition \ref{no tilting}.

\begin{proof}[Proof of Proposition \ref{no tilting}]
(a) is clear. (b) is shown in Theorems \ref{general tilting for qgr}.
(c) and (f) are well known.
(d) and the first sentence of (e) are shown in Theorems \ref{general tilting} ($n=1$) and \ref{negative case} ($n\ge2$).
The last sentence of (e) follows from Proposition \ref{k[w]/(w^2)}.
\end{proof}

\end{document}